\documentclass[12pt,a4paper,intlimits,sumlimits]{amsart}

\usepackage{amsmath, amsthm, mathtools}
\usepackage{amsfonts}
\usepackage[alphabetic,nobysame]{amsrefs}
\usepackage{graphicx}
\usepackage{framed}
\usepackage{amssymb}
\usepackage{esvect}
\usepackage{xcolor}
\usepackage{eucal}

\usepackage{hyperref}
\usepackage[english]{babel}

\def \N {\mathbb{N}}
\def \R {\mathbb{R}}

\newcommand{\tr}{\operatorname{tr}}

\newtheorem{theorem}{Theorem}[section]
\newtheorem{proposition}[theorem]{Proposition}
\newtheorem{lemma}[theorem]{Lemma}
\newtheorem{corollary}[theorem]{Corollary}
\newtheorem{openpb}[theorem]{Open Problem}

\theoremstyle{definition}

\numberwithin{equation}{section}

\usepackage{geometry}
\renewcommand{\epsilon}{\varepsilon}

\renewcommand{\le}{\leqslant}

\renewcommand{\ge}{\geqslant}

\allowdisplaybreaks

\title[Kernel representations with constant coefficients]{Kernel representations \\ of fractional gradient bilinear forms.\\
Part I. Constant coefficients}

\author{Serena Dipierro, Caterina Sportelli, and Enrico Valdinoci}

\address{Serena Dipierro: Department of Mathematics and Statistics, The University of Western Australia, 35 Stirling Highway, Crawley, Perth, WA 6009, Australia}
\email{serena.dipierro@uwa.edu.au}

\address{Caterina Sportelli: Departamento de Análisis Matemático, Universidad de Granada, 18071 Granada, Spain  \& Department of Mathematics and Statistics, The University of Western Australia, 35 Stirling Highway, Crawley, Perth, WA 6009, Australia}
\email{caterina.sportelli@uwa.edu.au, caterina.sp@ugr.es}

\address{Enrico Valdinoci: Department of Mathematics and Statistics, The University of Western Australia, 35 Stirling Highway, Crawley, Perth, WA 6009, Australia}
\email{enrico.valdinoci@uwa.edu.au}

 \date{}

\begin{document}

\begin{abstract}
We study the representation of bilinear forms associated with Riesz fractional gradients by symmetric integral kernels. Motivated by \cite[{Open Problem 1.10}]{MR3819528}, we
give a complete description of the constant coefficient case. We prove that any symmetric representing kernel is uniquely determined and we derive its explicit formula in terms of the constant matrix~$A$. This yields a sharp characterization of the matrices for which the kernel is either
nonnegative or uniformly positive. 
In particular, \cite[{Open Problem 1.10}]{MR3819528} has a negative answer (even for constant matrices). We also identify a class of nonconstant antisymmetric divergence free perturbations that are invisible to the fractional bilinear form and hence leave the representing kernel unchanged. The general case of nonconstant matrices exhibits a substantially different behavior and will be the main subject of our forthcoming paper \cite{DSVPART2}.
\end{abstract}

\maketitle

\tableofcontents

\section{Introduction} \subsection{{F}rom fractional gradients to kernel representations}
The recent literature has intensively analyzed the foundational properties of Riesz fractional gradient vectors. Specifically, given~$s\in(0,1)$,
for all~$u$ in the Schwartz space~$\mathcal S(\R^N)$
of smooth and rapidly decreasing functions, one defines\footnote{Throughout the paper, given $u\in\mathcal S(\R^N)$, we denote its Fourier transform by either $\widehat{u}(\xi)$ or $\mathcal F(u)(\xi)$
and adopt the following normalization:
\begin{equation*}
\widehat{u}(\xi):=(2\pi)^{-\frac{N}{2}}\int_{\R^N}u(x) \,e^{-ix\cdot\xi} dx.
\end{equation*}
With this convention, Plancherel's identity reads as
\begin{equation*}
 \int_{\R^N}u(x)\overline{v(x)}\,dx=\int_{\R^N}\widehat{u}(\xi)\overline{\widehat{v}(\xi)}\,d\xi.
\end{equation*}}
the fractional gradient~$D^s u=(D_1^su,\ldots,D_N^su)$ as
\begin{equation}\label{defnfrgradient}
 \widehat{D_j^s u}(\xi)=i\xi_j|\xi|^{s-1}\widehat{u}(\xi),\quad \mbox{ for } j=1,\ldots,N.
\end{equation}
The interest in these objects within fractional calculus and nonlocal equations is driven by their dual vector properties, divergence structures, and robust framework for handling distributions and functions of bounded variation. They also play a pivotal role in establishing Sobolev, Hardy, Trudinger and Morrey inequalities, $\Gamma$-convergence, fractional nonlinear operators and fractional Fredholm alternatives, as highlighted by recent results in~\cite{MR3403430, MR3420498, MR3615452, MR3714833, MR3819528, MR4001075, MR4048032, MR4104831, MR4179861, MR4895172}. 

Additionally, a common methodology in fractional problems is to link nonlocal operators to integral or integrodifferential operators, bridging Fourier analysis and potential theory. A critical step in this approach involves representing the operator via an interaction kernel and identifying its core singular structure.  Nonlocal operators in divergence form and their associated interaction kernels have recently been studied in \cite{MR5007887},  with a focus on existence theory for integrable data and the limit $s\nearrow1$ toward classical operators in divergence form.

To efficiently advance the analysis of fractional equations, it is often essential to establish sharp bounds on this kernel, typically inspired
by the structure exhibited by the fractional Laplacian.

In this spirit, the following foundational question was posed in~\cite[{Open Problem 1.10}]{MR3819528}:

\begin{openpb}\label{opeSPJMDPKM}
Let $s\in(0,1)$. Let $A:\R^N\to\R^{N\times N}$ be bounded and measurable. Assume that~$A$ is uniformly elliptic, namely that there exist constants $0<\lambda_A\le\Lambda_A<+\infty$ such that, for every~$x, \xi\in\R^N$,
\[
\lambda_A |\xi|^2\le A(x)\xi\cdot\xi\le\Lambda_A |\xi|^2.
\]
Can one find a nonnegative kernel $k_A$ such that
\begin{equation}\label{lòkgyftdresdetfyguhijok}
k_A(x,y)=k_A(y,x),
\end{equation}
and, for some $\alpha\in(0,2)$, $\eta>0$, $0<\lambda\le\Lambda<+\infty$, and~$M>0$,
\begin{alignat}{2}\label{srdtfyguioplèòlkoijhugyftr1}
&\lambda\le k_A(x,y)|x-y|^{N+\alpha}\le\Lambda,\quad&&\mbox{ for } |x-y|\le 1,\\ \label{srdtfyguioplèòlkoijhugyftr2}
\mbox{ and }\quad &k_A(x,y)\le M|x-y|^{-N-\eta},\quad&&\mbox{ for } |x-y|>1,
\end{alignat}
such that
\begin{equation}\label{mnbhgvcfdxszxdfcgvhb}
\int_{\R^N} A(x)D^s u(x)\cdot D^s v(x)\,dx = \iint_{\R^{2N}}(u(x)-u(y))(v(x)-v(y))k_A(x,y)\,dx\,dy
\end{equation}
for every~$u,v\in C_c^{\infty}(\R^N)$?
\end{openpb}
We show that the answer to Open Problem \ref{opeSPJMDPKM} is negative in general, both for matrices with variable coefficients (as it will be shown in the forthcoming paper~\cite{DSVPART2}) and, in particular, even for constant matrices~$A$.

More precisely, we characterize the constant matrices for which a kernel satisfying the required bounds exists.

We also refer the interested reader to the very recent work \cite{Carlotti26}, which considers the constant symmetric positive definite case in connection with the Markov property of the associated form and the resulting $L^p$ properties. Our focus is instead on the representation theory of the kernel itself, including its uniqueness, structural properties, and quantitative bounds.

\subsection{Kernel representation theory for constant matrices}
Our main result provides a complete characterization of the constant symmetric matrices for which the bilinear form associated with~$A$ admits a representation through a symmetric kernel satisfying the required ellipticity bounds. 
We first identify the unique kernel representing the bilinear form, and then we derive an explicit formula for it (see Propositions \ref{thmkernelstruct} and \ref{thmmainconst}, respectively). This formula yields a necessary and sufficient condition for the kernel to satisfy the required ellipticity bounds, namely,
\begin{equation}\label{tracecondition}
(N+2s)\lambda_{\min}(A)>\tr(A),
\end{equation}where~$\lambda_{\min}(A)$ is the smallest eigenvalue of the matrix~$A$.
The inequality in~\eqref{tracecondition} allows us to provide a complete characterization of the constant coefficient case of Open Problem~\ref{opeSPJMDPKM}, as stated in the forthcoming Theorem~\ref{corcharactA}.
To state our results, we use the notation
$$\Delta:=\{(x,x):x\in\R^N\}.$$

\begin{theorem}\label{corcharactA}
Let $s\in(0,1)$, $N\ge 1$ and $A\in\R^{N\times N}$ be a constant and symmetric matrix\footnote{
Let $A\in\R^{N\times N}$ be a constant matrix, and denote by
\[
A_{sym}:=\frac{A+A^T}{2},\quad A_{asym}:=\frac{A-A^T}{2}
\]
its symmetric and antisymmetric parts, respectively.

We notice that the bilinear form associated with~$A$ depends only on $A_{sym}$. Indeed, using the Fourier representation of the fractional gradient, we obtain that
\[
\int_{\R^N}A_{asym}D^su(x)\cdot D^sv(x)\,dx = C_{N,s} \int_{\R^N}|\xi|^{2s-2}\left(A_{asym}\xi\cdot\xi\right)
\widehat u(\xi)\widehat v(-\xi)\,d\xi.
\]
Since $A_{asym}^T=-A_{asym}$, we have that $A_{asym}\xi\cdot\xi=0$ for every~$\xi\in\R^N$. Therefore,
\[
\int_{\R^N} A D^su(x)\cdot D^sv(x)\,dx = \int_{\R^N}A_{sym}D^su(x)\cdot D^sv(x)\,dx.
\]
Hence, in the case of constant coefficient, we may assume without loss of generality that~$A$ is symmetric.

In addition, we have that $z\cdot Az=z\cdot A_{sym}z$ and~$\tr(A)=\tr(A_{sym})$,  so that the explicit representing kernel associated with~$A$ is the
same as that associated with $A_{sym}$.}. Let $k_A\in L^1_{loc}(\R^{2N}\setminus\Delta)$ and assume that~\eqref{lòkgyftdresdetfyguhijok} holds.

Then, 
the following statements are equivalent:
\begin{enumerate}
\item[(i)] 
For every~$u,v\in\mathcal S(\R^N)$,
\begin{equation*}
\int_{\R^N}A D^su(x)\cdot D^sv(x)\,dx = \iint_{\R^{2N}} (u(x)-u(y))(v(x)-v(y)) k_A(x, y)\,dx\,dy.
\end{equation*}

\item[(ii)]
For almost every~$x\ne y$,
\[
k_A(x,y) = \frac{c_{N,s}}{4s|x-y|^{N+2s}}\left((N+2s)\frac{(x-y)\cdot A(x-y)}{|x-y|^2}-\tr(A)\right).
\]
\end{enumerate}

Also, if either~$(i)$ or~$(ii)$ holds true, then
\begin{itemize}
\item \eqref{srdtfyguioplèòlkoijhugyftr2} is satisfied for all~$\eta\in(0,2s]$ and a suitable constant $M>0$,
 \item \eqref{srdtfyguioplèòlkoijhugyftr1} holds for some $\alpha\in(0,2)$ and~$0<\lambda\le\Lambda<\infty$,  if and only if \eqref{tracecondition} holds, and in this case necessarily~$\alpha=2s$.
\end{itemize}
\end{theorem}

We point out that, in the setting of Theorem~\ref{corcharactA}, 
one can pick the optimal exponent~$\eta=2s$
for~\eqref{srdtfyguioplèòlkoijhugyftr2} to be satisfied.

We further show that the representing kernel cannot have finite interaction range, even in the case of nonconstant matrices $A=A(x)$. More precisely, we have the following:

\begin{proposition}\label{propsuppcompatto}
Let $s\in(0,1)$ and $A:\R^N\to\R^{N\times N}$ be a bounded and uniformly elliptic matrix. Suppose that there exists a kernel
\[
k_A\in L^1_{\mathrm{loc}}(\R^{2N}\setminus\Delta)
\]
such that, for some $\alpha\in(0,2)$, $\eta>0$, and $\Lambda,M>0$,
\begin{alignat}{2}
k_A(x,y)&\le \Lambda |x-y|^{-N-\alpha}&&\quad\mbox{ for } |x-y|\le 1, \label{poiuytresdxcvghbnjm}\\
k_A(x,y)&\le M |x-y|^{-N-\eta}&&\quad\mbox{ for } |x-y|>1. \label{cxzsedrftyhujiko098765r4e3wq2a<}
\end{alignat}
Moreover, assume that \eqref{mnbhgvcfdxszxdfcgvhb} holds for every~$u\in C_c^{\infty}(\R^N)$.

Then, for every~$R_0>0$ there exist~$(x,y)\in\R^{2N}$ with $|x-y|>R_0$
such that~$k_A(x,y)\ne0$.
\end{proposition}
We will, in fact, provide a more general
and quantitative statement of
this result in~\cite{DSVPART2}.
Here, for the sake of simplicity and completeness, we present a self-contained proof
of Proposition~\ref{propsuppcompatto} in Appendix~\ref{appendixsuppcompatto}.

\subsection{Some results about nonconstant matrices}
We now ask whether the conclusions obtained for constant matrices may also hold for nonconstant matrices $A=A(x)$. This question will be the main topic of our forthcoming paper \cite{DSVPART2}.

In general, the argument based on the Fourier representation used for the case of constant matrices does no longer work. Nevertheless, there exists a nontrivial class of variable perturbations which are ``invisible" to the bilinear form.

More precisely, we consider perturbations of the form
\[
A(x)=A_0+C(x),
\]
where $A_0$ is constant and symmetric, while $C$ is an antisymmetric and divergence free matrix. Here, the divergence free condition\footnote{Interestingly,
divergence free matrices and tensors
appear naturally in a variety of concrete
applications, see e.g.~\cite{MR3813963, MR4449866}.} is meant by column, namely
\begin{equation}\label{òldKMAFJNBHCJSNKM,}
\sum_{i=1}^N \partial_i C_{ij}=0,\quad j=1,\dots,N.
\end{equation}
We show that such a perturbation does not contribute to the bilinear form. Indeed, the form associated with $A(x)$ coincides identically with that associated with $A_0$. As a consequence, not only the representation formula, but also the representing kernel and $m_A(\theta)$ are exactly the same as in the constant coefficient case.
\begin{theorem}\label{invisiblepert}
Let $s\in(0,1)$ and~$N\ge 1$. Let $A_0\in\R^{N\times N}$ be a constant and symmetric matrix. Assume that $C\in C_c^{\infty}(\R^N, \R^{N\times N})$ satisfies
\begin{equation}\label{lpèkoijhugyftdrxcfv}
C(x)^T=-C(x)\quad\mbox{ for every }x\in\R^N,
\end{equation}
and~\eqref{òldKMAFJNBHCJSNKM,} holds.

Let $A:\R^N\to\R^{N\times N}$ be defined as
\[
A(x):=A_0+C(x).
\]
Then, for every~$u,v\in\mathcal S(\R^N)$,
\begin{equation}\label{òè+poiuhygftdrszedrftychvghib}
\int_{\R^N} A(x)D^su(x)\cdot D^sv(x)\,dx = \int_{\R^N} A_0 D^su(x)\cdot D^sv(x)\,dx.
\end{equation}
In particular, the kernel representing the bilinear form associated with~$A(x)$ is the same as the one associated with $A_0$, namely
\[
k_A(x,y)=k_{A_0}(x,y)\quad\mbox{ for a.e.  }x\ne y,
\]
where
\[
k_{A_0}(x,y) = \frac{c_{N,s}}{4s|x-y|^{N+2s}}\left((N+2s)\frac{(x-y)\cdot A_0(x-y)}{|x-y|^2}-\tr(A_0)\right).
\]
\end{theorem}

\subsection{Examples and counterexamples}
For constant symmetric matrices $A$,  inequality \eqref{tracecondition} is precisely the condition which guarantees a positive uniform lower bound for the kernel, as requested in~\eqref{srdtfyguioplèòlkoijhugyftr1}. 
This is made explicit through the following example.
Let $a\ge 0$ and consider the family of matrices
\[
A_a:=
\begin{pmatrix}
1+\frac{2a}{3} & -\frac{a}{3} & -\frac{a}{3}\\[5pt]
-\frac{a}{3} & 1+\frac{2a}{3} & -\frac{a}{3}\\[5pt]
-\frac{a}{3} & -\frac{a}{3} & 1+\frac{2a}{3}
\end{pmatrix}.
\]
In Section \ref{MS-34or2}, we will prove that \eqref{tracecondition} is verified for $a<s$. Then, we consider two examples, corresponding to $a:=s/2$ and $a:=s$, respectively, which exhibit a difference in the behavior of the kernel. Indeed, in the first case, the kernel admits a uniform positive lower bound
in~\eqref{srdtfyguioplèòlkoijhugyftr1}
and fulfills all the properties listed in Open Problem~\ref{opeSPJMDPKM},
whereas in the second case the kernel becomes negative. These examples will be analyzed in detail in Section \ref{MS-34or2}.
An additional counterexample, for which we thank Daniel E. Spector,  is provided in Appendix~\ref{appDanielExample}.

\subsection{Organization of the paper}
The rest of the paper is organized as follows. Section~\ref{PRE} collects preliminary results connecting the theory of fractional gradients with that of integral kernels. These results lead to Propositions~\ref{thmkernelstruct} and~\ref{thmmainconst}, which are then used to prove Theorem~\ref{corcharactA} in Section~\ref{lpkoijuhgyftdresdtfyguhijhb213}.

The case of nonconstant matrices is presented in Section~\ref{MS-34or}, and a number of specific examples are discussed in Section~\ref{MS-34or2}.
The article ends with some appendices collecting ancillary results
to be used in the proofs of the main theorems.

\section{Preliminary results towards the proof of Theorem \ref{corcharactA}}\label{PRE}

The proof of Theorem~\ref{corcharactA} requires several preliminary steps, which are developed in this section.

We begin by identifying the structural constraints imposed on any symmetric kernel representing the bilinear form associated with a constant matrix~$A$. Indeed, our first result shows that the translation invariance and scaling properties force every such kernel to be homogeneous of order $-N-2s$ and to depend on $x$ and~$y$ only through the direction of $x-y$, as stated in the following result:
\begin{proposition}\label{thmkernelstruct}
Let $s\in(0,1)$, $N\ge 1$ and $A\in\R^{N\times N}$ be a constant and symmetric matrix. Assume that there exists a kernel
\[
k_A\in L^1_{\mathrm{loc}}(\R^{2N}\setminus\Delta)
\]
satisfying \eqref{lòkgyftdresdetfyguhijok}. Assume that \eqref{mnbhgvcfdxszxdfcgvhb} holds for every~$u,v\in\mathcal S(\R^N)$.

Then, there exists a unique
function $m_A\in L^1(\mathbb S^{N-1})$ such that $m_A(\theta)=m_A(-\theta)$, and
\begin{equation}\label{kernelconmATBD}
k_A(x,y)=\frac{1}{|x-y|^{N+2s}}m_A\left(\frac{x-y}{|x-y|}\right)\quad\mbox{ for a.e. } (x, y)\in \R^{2N}\setminus\Delta.
\end{equation}
\end{proposition}

Notice that Proposition~\ref{thmkernelstruct} determines the homogeneity of any representing kernel.  Therefore,  we need to compare its representation with~\eqref{srdtfyguioplèòlkoijhugyftr1} and~\eqref{srdtfyguioplèòlkoijhugyftr2}.
This yields necessary restrictions on the exponents in~\eqref{srdtfyguioplèòlkoijhugyftr1} and~\eqref{srdtfyguioplèòlkoijhugyftr2}, and on the function $m_A(\theta)$ provided by Proposition~\ref{thmkernelstruct}, as stated in the next result.
\begin{corollary}\label{corkernelbounds}
Let $s\in(0,1)$, $N\ge 1$ and $A\in\R^{N\times N}$ be a constant and symmetric matrix. Assume that there exists a kernel
\[
k_A\in L^1_{loc}(\R^{2N}\setminus\Delta)
\] satisfying
\eqref{lòkgyftdresdetfyguhijok}. Assume that \eqref{mnbhgvcfdxszxdfcgvhb} holds for every~$u,v\in\mathcal S(\R^N)$.

In addition,  suppose that there exist
\[
\alpha\in(0,2),\quad\eta>0,\quad 0<\lambda\le\Lambda<+\infty,\quad M>0,
\]
such that~\eqref{srdtfyguioplèòlkoijhugyftr1} and~\eqref{srdtfyguioplèòlkoijhugyftr2} are verified for almost every~$x\ne y$.

Then, 
\[
\alpha=2s,\quad 0<\eta\le2s, \quad\mbox{ and }\quad \lambda\le m_A(\theta)\le\min\{\Lambda,M\} \ \ \mbox{for a.e. }\theta\in\mathbb S^{N-1},
\]
where $m_A\in L^1(\mathbb S^{N-1})$ is the function provided by Proposition~\ref{thmkernelstruct}. 
\end{corollary}

We point out that Proposition~\ref{thmkernelstruct} and Corollary~\ref{corkernelbounds} are purely structural results, as they describe the form that a representing kernel must have. However, they do not identify the function $m_A(\theta)$ in terms of the matrix~$A$. 

As a further step, we compute explicitly the kernel $k_A(x, y)$. For constant symmetric matrices, the resulting formula gives a kernel of the form anticipated by Proposition~\ref{thmkernelstruct} and represents the bilinear form exactly.
\begin{proposition}\label{thmmainconst}
Let $s\in(0,1)$, $N\ge 1$ and $A\in\R^{N\times N}$ be a constant and symmetric matrix.  Let $k_A:\R^{2N}\setminus\Delta\to\R$ be defined as
\[
k_A(x,y):=\frac{c_{N, s}}{4s|x-y|^{N+2s}}\left((N+2s)\frac{(x-y)\cdot A(x-y)}{|x-y|^2}-\tr(A)\right).
\]
Then,  $k_A$ satisfies~\eqref{lòkgyftdresdetfyguhijok}, and, for any $u,v\in\mathcal S(\R^N)$, the identity in~\eqref{mnbhgvcfdxszxdfcgvhb} holds true.
\end{proposition}

We now collect some preliminary results that will be used in the proofs of Propositions~\ref{thmkernelstruct} and~\ref{thmmainconst}, which will
be presented, respectively, on pages~\pageref{PALKMND:01p}
and~\pageref{PALKMND:02p}.

\subsection{Fourier representation of the fractional gradient bilinear form}
Let $A\in\R^{N\times N}$ be a constant matrix. For $u,v\in\mathcal S(\R^N)$, we set
\begin{equation}\label{BAdef1}
 \mathcal{B}_A(u,v):=\int_{\R^N} A D^su(x)\cdot D^sv(x)\,dx.
\end{equation}
To start with, we derive the Fourier representation of $\mathcal{B}_A$.
\begin{proposition}\label{BAsymbol}
Let $A\in\R^{N\times N}$ be a constant matrix and let $\mathcal{B}_A$ defined as in~\eqref{BAdef1}. 

Then,
\begin{equation}\label{BAfourier}
 \mathcal{B}_A(u,v)=\int_{\R^N}p_A(\xi)\widehat{u}(\xi)\overline{\widehat{v}(\xi)}\,d\xi,
\end{equation}
where, for any $\xi\in\R^N$, 
\begin{equation}\label{ncewuri4y3itiyuwhdwuiefdefu5i654}
 p_A(\xi):=|\xi|^{2s-2}\,\xi\cdot A\xi.
\end{equation}
\end{proposition}

\begin{proof}
By using~\eqref{defnfrgradient}
in combination with Plancherel's identity, and recalling that $D^s_j\in L^2(\R^N)$ for any $j\in\{1,\dots, N\}$, we infer that
\[
\begin{split}
&\int_{\R^N}D_k^s u(x) D_j^s v(x) dx =\int_{\R^N}D_k^s u(x) \overline{D_j^s v(x)} dx =\int_{\R^N} \widehat{D_k^s u}(\xi)\overline{\widehat{D_j^s v}(\xi)}d\xi\\
&\qquad=\int_{\R^N}\bigl(i\xi_k|\xi|^{s-1}\widehat{u}(\xi)\bigr)\overline{\bigl(i\xi_j|\xi|^{s-1}\widehat{v}(\xi)\bigr)}\,d\xi=\int_{\R^N}\xi_k\xi_j|\xi|^{2s-2}\widehat{u}(\xi)\overline{\widehat{v}(\xi)}\,d\xi.
\end{split}
\]
Accordingly, by recalling~\eqref{BAdef1}, we get
\[
\mathcal{B}_A(u,v) =\sum_{j,k=1}^N A_{jk} \int_{\R^N}D_k^su(x) D_j^sv(x)\,dx=\int_{\R^N}|\xi|^{2s-2} \left(\sum_{j,k=1}^N A_{jk}\xi_j\xi_k\right)\widehat{u}(\xi)\overline{\widehat{v}(\xi)}\,d\xi,
\]
which is the desired result.
\end{proof}

\subsection{The structure of the kernel $k_A(x, y)$}
The proof of Proposition~\ref{thmkernelstruct} requires some intermediate steps. The next result guarantees the uniqueness of the representing kernel.
\begin{lemma}\label{lemmauniq}
Let $k_1,k_2\in L^1_{\mathrm{loc}}(\R^{2N}\setminus\Delta)$ satisfy~\eqref{lòkgyftdresdetfyguhijok}.  Assume that, for every~$u,v\in\mathcal S(\R^N)$, the integrals 
\[
\iint_{\R^{2N}} (u(x)-u(y))(v(x)-v(y))k_i(x,y)\,dx\,dy, \quad i\in\{1,2\}, 
\]
are absolutely convergent and satisfy
\[
\iint_{\R^{2N}} (u(x)-u(y))(v(x)-v(y))k_1(x,y)\,dx\,dy = \iint_{\R^{2N}} (u(x)-u(y))(v(x)-v(y))k_2(x,y)\,dx\,dy. 
\]

Then,
\[
k_1=k_2\quad\mbox{ a.e. in }\R^{2N}\setminus\Delta.
\]
\end{lemma}

\begin{proof}
We set $h:=k_1-k_2$.  Then, $h\in L^1_{loc}(\R^{2N}\setminus\Delta)$ and~$h(x, y)=h(y, x)$ for almost every~$(x,y)\in\R^{2N}\setminus\Delta$. Moreover, the absolute convergence assumption implies that, for every~$u,v\in\mathcal S(\R^N)$, 
\[
(u(x)-u(y))(v(x)-v(y))h(x,y) \in L^1(\R^{2N}).
\]
Hence, taking the difference of the two absolutely convergent integrals, we obtain that
\begin{equation}\label{defrbgrnjkwdlqòshbasnjk}
\iint_{\R^{2N}}(u(x)-u(y))(v(x)-v(y))h(x,y)\,dx\,dy=0
\end{equation}
for every~$u,v\in\mathcal S(\R^N)$.

Let $U,V\subset\R^N$ be open sets such that $\overline U\cap\overline V=\varnothing$, and take $u\in C_c^{\infty}(U)$ and~$v\in C_c^{\infty}(V)$. Therefore, by~\eqref{defrbgrnjkwdlqòshbasnjk}, we get
\[
\begin{split}
0&=\iint_{\R^{2N}}(u(x)-u(y))(v(x)-v(y))h(x,y)\,dx\,dy\\
&=-\iint_{\R^{2N}}u(x)v(y)h(x,y)\,dx\,dy-\iint_{\R^{2N}}u(y)v(x)h(x,y)\,dx\,dy.
\end{split}
\]
Notice that the two integrals on the right hand side are absolutely convergent. Indeed, their integrands are supported in $\mbox{supp}(u)\times \mbox{supp}(v)$ and~$\mbox{supp}(v)\times \mbox{supp}(u)$, respectively,  which are compact subsets of $\R^{2N}\setminus\Delta$. Thus, their integrability follows from the fact that $h\in L^1_{loc}(\R^{2N}\setminus\Delta)$. 

By using the symmetry of $h$ and exchanging $x$ and~$y$ in the second integral, the two terms on the right hand side are equal.  Hence, for every~$u\in C_c^{\infty}(U)$ and~$v\in C_c^{\infty}(V)$,
\begin{equation}\label{èplqsdkowABJNCODKMsajcndbvzjknlkò}
\iint_{U\times V} u(x)v(y)h(x,y)\,dy\,dx=0.
\end{equation}

Since $h\in L^1_{loc}(U\times V)$, it defines a distribution $T_h\in\mathcal D'(U\times V)$ such that
\[
T_h(\varphi):=\iint_{U\times V} \varphi(x,y)\,h(x,y)\,dx\,dy,\quad{\mbox{for all }}\,\varphi\in C_c^{\infty}(U\times V).
\]
Accordingly, by~\eqref{èplqsdkowABJNCODKMsajcndbvzjknlkò}, we get
\[
T_h\left(u(x)v(y)\right)=0
\]
for every~$u\in C_c^{\infty}(U)$ and~$v\in C_c^{\infty}(V)$.

The uniqueness part of \cite[Theorem~5.1.1]{MR1996773}, applied to the zero distributions in $\mathcal D'(U)$ and~$\mathcal D'(V)$, implies that the only distribution~$T\in\mathcal D'(U\times V)$ satisfying
\[
T\left(u(x)v(y)\right)=0\quad\mbox{ for all } u\in C_c^{\infty}(U),\ v\in C_c^{\infty}(V),
\]
is the zero distribution. Hence
\[
T_h=0\quad\mbox{ in }\mathcal D'(U\times V).
\]
Since $T_h$ is the distribution induced by the locally integrable function $h$,  it follows that
\begin{equation}\label{vfcdxszaesrdtfghj}
h=0\quad\mbox{ a.e. in }U\times V.
\end{equation}

We now extend~\eqref{vfcdxszaesrdtfghj} to the whole $\R^{2N}\setminus\Delta$. To accomplish this goal, for every~$m\in\N$, we set
\[
K_m:=\left\{(x,y)\in\R^{2N} \mbox{ s.t. } |x|\le m,\ |y|\le m,\ |x-y|\ge \frac1m\right\}.
\]
Then $K_m$ is compact, $K_m\subset K_{m+1}$, and
\[
\R^{2N}\setminus\Delta=\bigcup_{m=1}^{+\infty} K_m.
\]
For every~$(x_0,y_0)\in K_m$, since $x_0\ne y_0$, there exist two open neighborhoods $U, V$ with $x_0\in U$, $y_0\in V$,  such that $\overline U\cap\overline V=\varnothing$. 

The sets $U\times V$ form an open cover of $K_m$. By compactness,  there exist finitely many such products $U_j\times V_j$ covering $K_m$. Since, thanks to~\eqref{vfcdxszaesrdtfghj}, $h=0$ a.e. in each $U_j\times V_j$, 
it follows that $h=0$ a.e. in $K_m$.
Since this holds for every~$m\in\N$, we conclude that $h=0$ a.e. in $\R^{2N}\setminus\Delta$, namely $k_1=k_2$ a.e. in $\R^{2N}\setminus\Delta$, as desired.
\end{proof}

We now show the translation invariance of any kernel satisfying~\eqref{mnbhgvcfdxszxdfcgvhb}.
\begin{lemma}\label{lemmatranslation}
Let $k_A\in L^1_{\mathrm{loc}}(\R^{2N}\setminus\Delta)$ satisfy~\eqref{lòkgyftdresdetfyguhijok}. Assume that~\eqref{mnbhgvcfdxszxdfcgvhb} holds for every~$u, v\in\mathcal S(\R^N)$.

Then, for every~$z\in\R^N$,
\[
k_A(x+z,y+z)=k_A(x,y)\quad\mbox{ for a.e.  }(x,y)\in\R^{2N}\setminus\Delta.
\]
\end{lemma}

\begin{proof}
For $z\in\R^N$, we set
\[
u_z(x):=u(x-z)\quad\mbox{ and }\quad v_z(x):=v(x-z).
\]
Since
\[
\mathcal F(u_z)(\xi)=e^{-iz\cdot\xi}\mathcal F(u)(\xi)\quad\mbox{ and }\quad \mathcal F(v_z)(\xi)=e^{-iz\cdot\xi}\mathcal F(v)(\xi),
\]
by Proposition~\ref{BAsymbol} we get
\[
\mathcal{B}_A(u_z,v_z)=\int_{\R^N}p_A(\xi)e^{-iz\cdot\xi}\widehat{u}(\xi) e^{iz\cdot\xi}\overline{\widehat{v}(\xi)}\,d\xi=\int_{\R^N}p_A(\xi)\widehat{u}(\xi)\overline{\widehat{v}(\xi)}\,d\xi=\mathcal{B}_A(u,v).
\]
Accordingly,  applying~\eqref{mnbhgvcfdxszxdfcgvhb} to $u_z, v_z$ and changing variables, we obtain that

\[
\begin{split}
\mathcal{B}_A(u,v)&=\mathcal{B}_A(u_z,v_z)=\iint_{\R^{2N}}\left(u(x-z)-u(y-z)\right)\left(v(x-z)-v(y-z)\right) k_A(x,y)\,dx\,dy\\
&=\iint_{\R^{2N}}(u(x)-u(y))(v(x)-v(y))k_A(x+z,y+z)\,dx\,dy.
\end{split}
\]
On this account, we see that~$k_A(x+z,y+z)$ is another symmetric kernel which represents $\mathcal B_A$.  Thus,  Lemma~\ref{lemmauniq} yields the desired result.
\end{proof}

We now prove an homogeneity property of any kernel
satisfying~\eqref{mnbhgvcfdxszxdfcgvhb}.
\begin{lemma}\label{lemmahomogn}
Let $k_A\in L^1_{\mathrm{loc}}(\R^{2N}\setminus\Delta)$ satisfy~\eqref{lòkgyftdresdetfyguhijok}. Assume that~\eqref{mnbhgvcfdxszxdfcgvhb} holds for every~$u, v\in\mathcal S(\R^N)$.

Then, for every~$r>0$,
\[
k_A(rx,ry)=r^{-N-2s}k_A(x,y)\quad\mbox{ for a.e.  }(x,y)\in\R^{2N}\setminus\Delta.
\]
\end{lemma}

\begin{proof}
Let $r>0$. We set
\[
u_r(x):=u(rx)\quad\mbox{ and }\quad v_r(x):=v(rx).
\]
Hence,
\[
\mathcal F(u_r)(\xi)=r^{-N}\mathcal F(u)\left(\frac{\xi}{r}\right) \quad\mbox{ and }\quad \mathcal F(v_r)(\xi)=r^{-N}\mathcal F(v)\left(\frac{\xi}{r}\right).
\]
By using Proposition~\ref{BAsymbol}, the change of variable~$\eta:=\xi/r$, and recalling that~$p_A(\xi)$ is homogeneous of degree~$2s$,  we infer that
\begin{equation}\label{jnkbhgvcfdxszfcgvhbjnkml,.}
\mathcal{B}_A(u_r,v_r)=\int_{\R^N}p_A(\xi)\widehat{u_r}(\xi)\overline{\widehat{v_r}(\xi)}\,d\xi=r^{-N}\int_{\R^N}p_A(r\eta)\widehat{u}(\eta)\overline{\widehat{v}(\eta)}\,d\eta=r^{2s-N}\mathcal{B}_A(u,v).
\end{equation}

On the one hand, applying~\eqref{mnbhgvcfdxszxdfcgvhb} to $u_r, v_r$, and using the changes of variable~$X:=rx, Y:=ry$, we obtain that
\begin{equation}\label{mkljnhbgvfcdxszxdcfvg0}
\begin{split}
\mathcal{B}_A(u_r,v_r)&=\iint_{\R^{2N}}\left(u(rx)-u(ry)\right)\left(v(rx)-v(ry)\right)k_A(x,y)\,dx\,dy\\
&=r^{-2N}\iint_{\R^{2N}}(u(X)-u(Y))(v(X)-v(Y))k_A\left(\frac{X}{r},\frac{Y}{r}\right)\,dX\,dY.
\end{split}
\end{equation}
On the other hand, by~\eqref{jnkbhgvcfdxszfcgvhbjnkml,.} and~\eqref{mnbhgvcfdxszxdfcgvhb} applied directly to $u$ and~$v$, 
\begin{equation}\label{mkljnhbgvfcdxszxdcfvg}
\mathcal{B}_A(u_r,v_r)=r^{2s-N}\mathcal{B}_A(u,v)=r^{2s-N}\iint_{\R^{2N}}(u(X)-u(Y))(v(X)-v(Y))k_A(X,Y)\,dX\,dY.
\end{equation}
Thus, by~\eqref{mkljnhbgvfcdxszxdcfvg0},~\eqref{mkljnhbgvfcdxszxdcfvg}, and Lemma~\ref{lemmauniq}, we deduce that
\[
k_A\left(\frac{X}{r},\frac{Y}{r}\right) = r^{N+2s}k_A(X,Y)\quad\mbox{ for a.e. } (X,Y)\in\R^{2N}\setminus\Delta.
\]
Setting $x:=X/r$ and~$y:=Y/r$, we obtain the desired result.
\end{proof}

We are now in the position to prove Proposition~\ref{thmkernelstruct}
and Corollary~\ref{corkernelbounds}.

\begin{proof}[Proof of Proposition~\ref{thmkernelstruct}]
We first show that there exists~$K_A\in L^1_{\mathrm{loc}}(\R^N\setminus\{0\})$ such that
\label{PALKMND:01p}
\begin{equation}\label{mnkjb hvcgfdxcghjnm}
\mbox{$k_A(x, y)=K_A(x-y)$}.
\end{equation}
To this end, we set~$\zeta:=x-y$ and~$q:=y$, and we define
\[
F(\zeta,q):=k_A(q+\zeta,q).
\]
Then, $F\in L^1_{loc}\left((\R^N\setminus\{0\})\times\R^N\right)$.  

We observe that, by Lemma~\ref{lemmatranslation},
for any~$z\in\R^N$,  
\[
k_A(q+\zeta+z,q+z)=k_A(q+\zeta,q)\quad\mbox{ for a.e.  }(\zeta,q)\in (\R^N\setminus\{0\})\times\R^N,
\]
namely,
\begin{equation}\label{lpkoijhugyftdreswrdtfgybhjnmkl,}
F(\zeta,q+z)=F(\zeta,q)\quad\mbox{ for a.e.  }(\zeta,q)\in (\R^N\setminus\{0\})\times\R^N.
\end{equation}
Thus, there exists a function $K_A\in L^1_{\mathrm{loc}}(\R^N\setminus\{0\})$ such that
\[
F(\zeta,q)=K_A(\zeta)\quad\mbox{ for a.e.  }(\zeta,q)\in (\R^N\setminus\{0\})\times\R^N.
\]
Returning to the original variables, we obtain~\eqref{mnkjb hvcgfdxcghjnm}.

By~\eqref{mnkjb hvcgfdxcghjnm} and Lemma~\ref{lemmahomogn}, we infer that, for every~$r>0$
and for a.e.~$(x,y)\in\R^{2N}\setminus\Delta$,
\begin{equation}\label{bfdcxszsxedrctfvgybhunijm}
K_A(r(x-y))= k_A(rx,ry) = r^{-N-2s}k_A(x, y)= r^{-N-2s}K_A(x-y).
\end{equation}
Thus, by Lemma~\ref{lemmatecnico2.2} (here applied with $\gamma:=N+2s$) we infer that there exists $m_A\in L^1(\mathbb S^{N-1})$ such that
\begin{equation}\label{mn bjhvcgxfszddxfcvgbhnjkm,lò.à-}
K_A(\rho\theta)=\rho^{-N-2s}m_A(\theta)\quad\mbox{ for a.e.  }(\rho,\theta)\in(0,+\infty)\times\mathbb S^{N-1}.
\end{equation}
This is equivalent to
\[
K_A(z) =\frac{1}{|z|^{N+2s}} m_A\left(\frac{z}{|z|}\right)\quad\mbox{ for a.e.  }z\in\R^N\setminus\{0\}.
\]
Since $k_A(x,y)=k_A(y,x)$ for almost every~$(x,y)$, we have $K_A(z)=K_A(-z)$ for almost every~$z\in\R^N$. Therefore, by~\eqref{mn bjhvcgxfszddxfcvgbhnjkm,lò.à-} it follows that $m_A(\theta)=m_A(-\theta)$ for a.e.  $\theta\in\mathbb S^{N-1}$.
Moreover, recalling~\eqref{mnkjb hvcgfdxcghjnm}, we conclude that~\eqref{kernelconmATBD} holds. Finally, the uniqueness follows by Lemma~\ref{lemmauniq}.
\end{proof}

\begin{proof}[Proof of Corollary~\ref{corkernelbounds}]
By Proposition~\ref{thmkernelstruct} and~\eqref{srdtfyguioplèòlkoijhugyftr1}, we infer that
\[
\lambda\le m_A\left(\frac{x-y}{|x-y|}\right)|x-y|^{\alpha-2s}\le\Lambda
\]
for almost every~$(x,y)$ such that $0<|x-y|\le 1$. Namely,
\[
\lambda\le m_A\left(\frac{z}{|z|}\right)|z|^{\alpha-2s}\le\Lambda\quad\mbox{ for a.e.  $z$, with $0<|z|\le 1$}.
\]
By writing $z:=\rho\theta$,
with $\rho\in(0,1]$ and~$\theta\in\mathbb S^{N-1}$,  we obtain that
\begin{equation}\label{+èpoi0u8y7t6r5e4w3aq<szdxfc}
\lambda\le m_A(\theta)\rho^{\alpha-2s}\le\Lambda\quad\mbox{ for a.e. } (\rho,\theta)\in(0,1]\times\mathbb S^{N-1}.
\end{equation}
The lower bound in~\eqref{+èpoi0u8y7t6r5e4w3aq<szdxfc} implies, in particular, that
\begin{equation}\label{nbhjgvcfdxsza<sxdcfvgb}
m_A(\theta)>0.
\end{equation}

Now, since~\eqref{+èpoi0u8y7t6r5e4w3aq<szdxfc} holds for almost every~$\rho\in(0,1)$, there exists an infinitesimal sequence~$\rho_j\subset(0,1)$ as~$j\to+\infty$
such that~\eqref{+èpoi0u8y7t6r5e4w3aq<szdxfc} holds at $\rho=\rho_j$ for every~$j\in\N$.

Assume first that $\alpha>2s$. Then,
\[
\lim_{j\to +\infty} m_A(\theta)\rho_j^{\alpha-2s} = 0,
\]
which contradicts the lower bound in~\eqref{+èpoi0u8y7t6r5e4w3aq<szdxfc}.

Suppose instead that $\alpha<2s$. By~\eqref{nbhjgvcfdxsza<sxdcfvgb}, we have
\[
\lim_{j\to +\infty} m_A(\theta)\rho_j^{\alpha-2s} = +\infty,
\]
which contradicts the upper bound in~\eqref{+èpoi0u8y7t6r5e4w3aq<szdxfc}.

From these considerations, we conclude that necessarily~$
\alpha=2s$.
Using this information in~\eqref{+èpoi0u8y7t6r5e4w3aq<szdxfc}, we get
\begin{equation}\label{lkjohiyguftdersawesrdtfyguhij}
\lambda\le m_A(\theta)\le\Lambda\quad\mbox{for a.e. }\theta\in\mathbb S^{N-1}.
\end{equation}

We now consider the case $|x-y|>1$. In this case, we use~\eqref{srdtfyguioplèòlkoijhugyftr2} and we see that
\begin{equation}\label{nhbgvfcdxsrdtfygbhujnkm}
m_A(\theta)\le M\rho^{2s-\eta} \quad\mbox{ for a.e. } (\rho,\theta)\in(1,+\infty)\times\mathbb S^{N-1},
\end{equation}
If $\eta>2s$, then we can choose a diverging sequence~$\rho_j\in(1,+\infty)$ as~$j\to+\infty$ such that~\eqref{nhbgvfcdxsrdtfygbhujnkm} holds at~$\rho=\rho_j$ for all~$j\in\N$.
Accordingly, by~\eqref{lkjohiyguftdersawesrdtfyguhij} and~\eqref{nhbgvfcdxsrdtfygbhujnkm}, we have that
\[
\begin{split}
&\lambda\le m_A(\theta)\le M\rho_j^{2s-\eta} \ \mbox{ for all } j\in\N, \\
\mbox{ and } \quad &\lim_{j\to +\infty} M\rho_j^{2s-\eta} =0,
\end{split}
\]
which is a contradiction. Hence,  $\eta\le 2s$.  Since $\eta>0$ by assumption, we infer that~$0<\eta\le 2s$.

Therefore, choosing a sequence~$\rho_j \in(1,+\infty)$ such that
\[
\lim_{j\to +\infty} \rho_j = 1
\]
and~\eqref{nhbgvfcdxsrdtfygbhujnkm} holds at $\rho=\rho_j$ for all $j\in\N$, we get $m_A(\theta)\le M$. 

Combining this with~\eqref{lkjohiyguftdersawesrdtfyguhij}, we conclude that
\[
\lambda\le m_A(\theta)\le\min\{\Lambda,M\} \quad\mbox{for a.e. }\theta\in\mathbb S^{N-1}.
\]
Gathering all these considerations, we obtain the desired claims.
\end{proof}

\subsection{Fourier representation of the bilinear form induced by $k_A(x, y)$}\label{PR3}

By Proposition~\ref{thmkernelstruct},
any symmetric kernel $k_A$ satisfying~\eqref{mnbhgvcfdxszxdfcgvhb} must have the form~\eqref{kernelconmATBD}. In light of this fact,
we will derive the Fourier representation of the bilinear form induced by $k_A(x, y)$ as follows.

Let $m\in L^1(\mathbb S^{N-1})$ be such that $ m(-\theta)=m(\theta)$ for almost every~$\theta\in\mathbb S^{N-1}$. We define
\begin{equation}\label{Emdef}
\begin{split}
\mathcal{E}_m(u,v)&:=\frac{c_{N,s}}{2}\iint_{\R^{2N}}\frac{(u(x)-u(y))(v(x)-v(y))}{|x-y|^{N+2s}}m\left(\frac{x-y}{|x-y|}\right)\,dx\,dy\\
&=\frac{c_{N,s}}{2}\iint_{\R^{2N}}\frac{(u(x)-u(x-h))(v(x)-v(x-h))}{|h|^{N+2s}}m\left(\frac{h}{|h|}\right)\,dx\,dh
\end{split}
\end{equation}
where
\[
c_{N,s}:=\frac{2^{2s}s\,\Gamma\!\left(\frac N2+s\right)}{\pi^{N/2}\Gamma(1-s)}>0.
\]
With this choice, we have the following:
\begin{lemma}\label{prop:E1normalization}
Let $s\in(0,1)$. Then, for every~$u,v\in\mathcal S(\mathbb R^N)$,
\begin{equation}\label{E1normalization}
\mathcal{E}_1(u,v)=\int_{\R^N}D^su(x)\cdot D^sv(x) dx.
\end{equation}
\end{lemma}

The proof of Lemma~\ref{prop:E1normalization} is standard, but we provide the details in Appendix~\ref{appe:E1normalization} for reader's convenience.

The next result guarantees that the integral in~\eqref{Emdef} is absolutely convergent on Schwartz functions. 

\begin{lemma}\label{lemmabsconv}
Let $u,v\in\mathcal S(\R^N)$. Then, there exists a constant $C=C(N, s)>0$ such that
\[
|\mathcal{E}_m(u,v)|\le C\,\|m\|_{L^1(\mathbb S^{N-1})}\left(\|\nabla u\|_{L^2(\R^N)}\|\nabla v\|_{L^2(\R^N)} +\|u\|_{L^2(\R^N)}\|v\|_{L^2(\R^N)} \right).
\]
\end{lemma}

\begin{proof}
We set $h\in\R^N$. By the Cauchy-Schwarz inequality, we get
\[
\int_{\R^N}|u(x)-u(x-h)|\,|v(x)-v(x-h)|\,dx\le\|u(\cdot)-u(\cdot-h)\|_{L^2(\R^N)}\|v(\cdot)-v(\cdot-h)\|_{L^2(\R^N)}.
\]
If $|h|\le 1$, then the fundamental theorem of calculus gives
\[
u(x)-u(x-h)=\int_0^1 h\cdot\nabla u(x-th)\,dt.
\]
Therefore,
\[
 \|u(\cdot)-u(\cdot-h)\|_{L^2(\R^N)}\le |h|\,\|\nabla u\|_{L^2(\R^N)},
\]
and the same estimate holds for $v$.  Hence,  passing to polar coordinates, we have 
\begin{equation}\label{pèkdowBH BJNKML}
\begin{split}
&\int_{|h|\le 1}\frac{|m(h/|h|)|}{|h|^{N+2s}}\int_{\R^N}|u(x)-u(x-h)|\,|v(x)-v(x-h)|\,dx\,dh\\
&\quad\le\|\nabla u\|_{L^2(\R^N)}\|\nabla v\|_{L^2(\R^N)}\int_{|h|\le 1}\left|m\left(\frac{h}{|h|}\right)\right|\,|h|^{2-N-2s}\,dh\\
&\quad=\|\nabla u\|_{L^2(\R^N)}\|\nabla v\|_{L^2(\R^N)} \int_{\mathbb S^{N-1}}|m(\omega)|\int_0^1 r^{2-N-2s}r^{N-1}\,dr\,d\omega\\
&\quad=\|\nabla u\|_{L^2(\R^N)}\|\nabla v\|_{L^2(\R^N)}\|m\|_{L^1(\mathbb S^{N-1})}\int_0^1 r^{1-2s}\,dr
<+\infty.
\end{split}
\end{equation}
If instead $|h|>1$, then
\[
\int_{\R^N}|u(x)-u(x-h)|\,|v(x)-v(x-h)|\,dx\le 4 \|u\|_{L^2(\R^N)}\|v\|_{L^2(\R^N)}.
\]
Thus, passing to polar coordinates, we have
\begin{equation}\label{esdr6tfygyhuijnkp}
\begin{split}
&\int_{|h|> 1}\frac{|m(h/|h|)|}{|h|^{N+2s}}\int_{\R^N}|u(x)-u(x-h)|\,|v(x)-v(x-h)|\,dx\,dh\\
&\quad\le 4 \|u\|_{L^2(\R^N)}\|v\|_{L^2(\R^N)} \int_{|h|> 1}\frac{|m(h/|h|)|}{|h|^{N+2s}}\, dh\\
&\quad\le 4 \|u\|_{L^2(\R^N)}\|v\|_{L^2(\R^N)} \|m\|_{L^1(\mathbb S^{N-1})}\int_1^{+\infty} r^{-1-2s}\,dr <+\infty.
\end{split}
\end{equation}
Thus, combining~\eqref{Emdef},~\eqref{pèkdowBH BJNKML} and~\eqref{esdr6tfygyhuijnkp}, we obtain the desired result.
\end{proof}

The next result will be useful to write~\eqref{Emdef} in Fourier variables.
\begin{lemma}\label{lemmafindiff}
For $u\in\mathcal S(\R^N)$ and~$h\in\R^N$, 
\begin{equation*}
\int_{\R^N}(u(x)-u(x-h))(v(x)-v(x-h))\,dx = 2\int_{\R^N}\left(1-\cos(h\cdot\xi)\right)\widehat{u}(\xi)\overline{\widehat{v}(\xi)}\,d\xi.
\end{equation*}
\end{lemma}

\begin{proof}
We recall that
\[
\mathcal F(u(\cdot)-u(\cdot-h))(\xi)=\bigl(1-e^{-ih\cdot\xi}\bigr)\mathcal F(u)(\xi).
\]
The desired result then follows
by applying
Plancherel's identity and using that $|1-e^{-it}|^2=2(1-\cos t)$ for any~$t\in\R$.
\end{proof}

Furthermore, we observe the following:
\begin{lemma}\label{lemmint}
Let $s\in (0, 1)$. Then, 
\begin{equation}\label{kappasconstant}
\kappa_s:=\int_0^{+\infty}\frac{1-\cos r}{r^{1+2s}}\,dr
\end{equation}
is finite and positive. 

Moreover, for any $t\in\R$, 
\begin{equation}\label{radscaling}
 \int_0^{+\infty}\frac{1-\cos(rt)}{r^{1+2s}}\,dr
 =\kappa_s|t|^{2s}.
\end{equation}
\end{lemma}

\begin{proof}
Since
\[
1-\cos r=2\sin^2 \left(\frac r2\right),
\]
and~$|\sin t|\le |t|$ for any $t\in\R$, we obtain that $0\le 1-\cos r\le \frac{r^2}{2}$. Hence, 
\[
0\le\frac{1-\cos r}{r^{1+2s}}\le\frac{r^{1-2s}}{2}.
\]
Since $s\in (0, 1)$, the right-hand side is integrable on $(0,1)$.

Moreover,  $0\le 1-\cos r\le2$ for $r>0$ and therefore
\[
0\le\frac{1-\cos r}{r^{1+2s}}\le\frac{2}{r^{1+2s}}.
\]
Since $s\in (0, 1)$, the right hand side is integrable on $(1, +\infty)$.
Therefore,  it follows that~$\kappa_s<+\infty$. Furthermore, since $1-\cos r\ge0$, and it is strictly positive for every~$r\notin2\pi\mathbb Z$, we infer that $\kappa_s>0$.

We now prove~\eqref{radscaling}.
If $t=0$, then~\eqref{radscaling} plainly follows.  If $t\ne0$, then the change of variable $\rho:=|t|r$ yields
\[
\int_0^{+\infty}\frac{1-\cos(rt)}{r^{1+2s}}\,dr=|t|^{2s}\int_0^{+\infty}\frac{1-\cos(\rho\,\mathrm{sgn}(t))}
{\rho^{1+2s}}\,d\rho =|t|^{2s}\int_0^{+\infty}\frac{1-\cos\rho}{\rho^{1+2s}}\,d\rho=\kappa_s|t|^{2s},
\]
as desired.
\end{proof}

Combining Lemmata \ref{lemmafindiff} and~\ref{lemmint}, we obtain an explicit expression for the Fourier representation of the bilinear form defined in~\eqref{Emdef}.
\begin{proposition}\label{propEmsymbol}
For $u,v\in\mathcal S(\R^N)$, we have that
\[
\mathcal{E}_m(u,v)=\int_{\R^N}\psi_m(\xi)\,\widehat{u}(\xi)\,\overline{\widehat{v}(\xi)}\, d\xi,
\]
where
\begin{equation}\label{lòkjhgfdtresdtfyguhijokpl}
\psi_m(\xi):=
\begin{cases}\displaystyle
c_{N,s}\kappa_s|\xi|^{2s}\int_{\mathbb S^{N-1}}m(\theta)\left|\frac{\xi}{|\xi|}\cdot\theta\right|^{2s}\,d\theta &{\mbox{ if~$\xi\neq 0$,}}\\
0 &{\mbox{ if~$\xi=0$,}}\end{cases}
\end{equation}
being $\kappa_s$ as in~\eqref{kappasconstant}.
\end{proposition}

\begin{proof}
By Lemma~\ref{lemmafindiff}, we obtain that
\[
\begin{split}
\mathcal{E}_m(u,v) &= \frac{c_{N,s}}{2}\int_{\R^N}\frac{m(h/|h|)}{|h|^{N+2s}}\int_{\R^N}(u(x)-u(x-h))(v(x)-v(x-h)) \,dx \,dh\\
&= c_{N,s}\int_{\R^N}\frac{m(h/|h|)}{|h|^{N+2s}}\int_{\R^N}(1-\cos(h\cdot\xi))\,\widehat{u}(\xi)\,\overline{\widehat{v}(\xi)}\, d\xi\, dh.
\end{split}
\]
We point out that the change of the order of integration is justified by Fubini's theorem. Indeed,
Lemma~\ref{lemmabsconv}, applied with $|m|$, gives the absolute convergence of the original double integral, and the above Fourier representation preserves this integrability. 

Hence, 
\[
\mathcal{E}_m(u,v) = c_{N,s}\int_{\R^N}\left(\;\int_{\R^N}\frac{m(h/|h|)}{|h|^{N+2s}}(1-\cos(h\cdot\xi))\,dh\right)\widehat{u}(\xi)\,\overline{\widehat{v}(\xi)}\,d\xi.
\]

We now use polar coordinate and write $h:=r\theta$, with $r>0$ and~$\theta\in\mathbb S^{N-1}$. Then,  by Lemma~\ref{lemmint}, we infer that
\[
c_{N,s}\int_{\mathbb S^{N-1}}m(\theta)\int_0^{+\infty} \frac{1-\cos(r\theta\cdot\xi)}{r^{1+2s}}\,dr\,d\theta = c_{N,s}\kappa_s\int_{\mathbb S^{N-1}} m(\theta)|\theta\cdot\xi|^{2s}\,d\theta.
\]
Thus, recalling the definition of~$\psi_m$ in~\eqref{lòkjhgfdtresdtfyguhijokpl} we obtain the desired result.
\end{proof}

We introduce the following quantity, which will provide the normalization constant in the Fourier representation in Proposition~\ref{propEmsymbol}.
\begin{lemma}\label{lemmauguale1}
For $\omega\in\mathbb S^{N-1}$, define
\begin{equation}\label{JNsdefn}
J_{N,s}(\omega):=\int_{\mathbb S^{N-1}}|\omega\cdot\theta|^{2s}\,d\theta.
\end{equation}
Then, $J_{N,s}(\omega)=J_{N,s}$ 
is independent of the direction~$\omega$
and it is strictly positive. 

Moreover,
\begin{equation}\label{xsrdcftvgbhnjkm,lkmjnbhgv}
 c_{N,s}\kappa_sJ_{N,s}=1,
\end{equation}
being $\kappa_s$ as in~\eqref{kappasconstant}.
\end{lemma}

\begin{proof}
We first show that $J_{N,s}$ is independent of the choice of $\omega\in\mathbb S^{N-1}$. To this end, let~$\omega_1,\omega_2\in\mathbb S^{N-1}$. Hence, there exists an orthogonal matrix
$R\in O(N)$ such that $R\omega_1=\omega_2$. Using the invariance of the
surface measure on $\mathbb S^{N-1}$ under rotations, we obtain that
\[
\int_{\mathbb S^{N-1}}|\omega_2\cdot\theta|^{2s}\,d\theta = \int_{\mathbb S^{N-1}}|R\omega_1\cdot\theta|^{2s}\,d\theta = \int_{\mathbb S^{N-1}}|\omega_1\cdot R^T\theta|^{2s}\,d\theta = \int_{\mathbb S^{N-1}}|\omega_1\cdot\theta|^{2s}\,d\theta.
\]
This shows that $J_{N,s}$ does not depend on $\omega$.

Moreover,  the integrand is nonnegative and not identically zero on $\mathbb S^{N-1}$, since $|\omega\cdot\theta|^{2s}>0$ for every~$\theta$ in a neighborhood of $\omega$. Therefore, $J_{N,s}>0$.

It only remains to prove~\eqref{xsrdcftvgbhnjkm,lkmjnbhgv}. From~\eqref{lòkjhgfdtresdtfyguhijokpl}, used here with $m\equiv 1$, we have that, for any~$\xi\in\R^N$ with $\xi\ne 0$,
\[
\psi_1(\xi) = c_{N,s}\kappa_s|\xi|^{2s}\int_{\mathbb S^{N-1}}\left|\frac{\xi}{|\xi|}\cdot\theta\right|^{2s}\,d\theta.
\]
By the independence of $J_{N,s}$ from the direction, we can write
\[
\psi_1(\xi) = c_{N,s}\kappa_sJ_{N,s}|\xi|^{2s}.
\]

Accordingly, the normalization in~\eqref{E1normalization} and Proposition~\ref{BAsymbol}, used here with $A:=I$, show that
\begin{eqnarray*}&&
\int_{\R^N} |\xi|^{2s}\,\widehat{u}(\xi)\overline{\widehat{v}(\xi)}\,d\xi = \mathcal{B}_{I}(u,v) =  \mathcal{E}_1(u,v) \\&&\qquad= \int_{\R^N}\psi_1(\xi)\widehat{u}(\xi)\overline{\widehat{v}(\xi)}\, d\xi
=
c_{N,s}\kappa_sJ_{N,s}\int_{\R^N} |\xi|^{2s}\,\widehat{u}(\xi)\overline{\widehat{v}(\xi)}\,d\xi ,
\end{eqnarray*}
which entails that~$c_{N,s}\kappa_sJ_{N,s}=1$, as desired.
\end{proof}

By combining Proposition~\ref{propEmsymbol} and Lemma~\ref{lemmauguale1}, we infer that, for any~$\xi\in\R^N$ 
with~$\xi\ne 0$, we can write
\begin{equation}\label{psinormalized}
\psi_m(\xi)=|\xi|^{2s} \ \dfrac{\displaystyle\int_{\mathbb S^{N-1}}m(\theta)\left|\frac{\xi}{|\xi|}\cdot\theta\right|^{2s}\,d\theta}{\displaystyle\int_{\mathbb S^{N-1}}\left|\frac{\xi}{|\xi|}\cdot\theta\right|^{2s}\,d\theta}.
\end{equation}

The following identity will allow us to relate the integral expression arising in~\eqref{psinormalized} to the trace of the matrix~$A$ and the quadratic form associated with~$A$.

\begin{lemma}\label{lemmaspherical}
Let $A\in\R^{N\times N}$ be symmetric and let $\omega\in\mathbb S^{N-1}$.  Then,
\begin{equation}\label{lsqkdJNCHUBSAKMOXxcsdjbhvhn}
 \frac{\displaystyle\int_{\mathbb S^{N-1}}(\theta\cdot A\theta)|\omega\cdot\theta|^{2s}\,d\theta}{\displaystyle\int_{\mathbb S^{N-1}}|\omega\cdot\theta|^{2s}\,d\theta}=\frac{\tr(A)+2s\,\omega\cdot A\omega}{N+2s}.
\end{equation}
\end{lemma}

\begin{proof}
We first consider the case $N=1$. In this case, $A=A_{11}$ and $\mathbb S^0=\{-1,1\}$. Hence, for every~$\omega,\theta\in\mathbb S^0$, we have $|\omega\theta|=1$ and $\theta\cdot A\theta=A_{11}$. Therefore,
\[
\frac{\displaystyle\int_{\mathbb S^0}(\theta\cdot A\theta)
|\omega\theta|^{2s}\,d\theta}
{\displaystyle\int_{\mathbb S^0}|\omega\theta|^{2s}\,d\theta}
=A_{11}.
\]
Moreover, since $|\omega|^2=1$, we have
\[
\frac{\tr(A)+2s\,\omega\cdot A\omega}{1+2s} = \frac{A_{11}+2sA_{11}}{1+2s} =A_{11}. 
\] The last two displays establish
\eqref{lsqkdJNCHUBSAKMOXxcsdjbhvhn} when $N=1$.

From now on we suppose that $N\ge 2$. Let $\omega\in\mathbb S^{N-1}$ and define
\begin{equation}\label{dretfyghuijoklpèòlppkoijuhgy}
F_\omega(A):=\int_{\mathbb S^{N-1}}(\theta\cdot A\theta)|\omega\cdot\theta|^{2s}\,d\theta =\sum_{i,j=1}^N A_{ij}\int_{\mathbb S^{N-1}}\theta_i\theta_j|\omega\cdot\theta|^{2s}\,d\theta.
\end{equation}
We claim that
\begin{equation}\label{pèkljniopjhgvftyrdszdwez}
F_\omega(A) = \frac{J_{N,s}}{N+2s}\left(\tr(A)+2s\,\omega\cdot A\omega\right).
\end{equation}
We first prove~\eqref{pèkljniopjhgvftyrdszdwez} in the case $\omega=e_1$.  By~\eqref{dretfyghuijoklpèòlppkoijuhgy},  we infer that
\begin{equation}\label{opkjoikjnio0ijohutfde4eswe4e5dstyh}
F_{e_1}(A)=\int_{\mathbb S^{N-1}}(\theta\cdot A\theta)|\theta_1|^{2s}\,d\theta =\sum_{i,j=1}^N A_{ij}\int_{\mathbb S^{N-1}}\theta_i\theta_j|\theta_1|^{2s}\,d\theta.
\end{equation}

Let $i\ne j$. We claim that
\begin{equation}\label{,lkjnhbgvfctdxrctfvgbhnm}
\int_{\mathbb S^{N-1}}\theta_i\theta_j|\theta_1|^{2s}\,d\theta=0.
\end{equation}
Indeed, since $i\neq j$, at least one of the two indices $i$ and~$j$ is different from $1$. Let $k\in\{i,j\}$ be such an index, and consider the reflection
\[
R_k(\theta_1,\dots,\theta_k,\dots,\theta_N) = (\theta_1,\dots,-\theta_k,\dots,\theta_N).
\]
The map $R_k$ preserves both the sphere and the surface measure. Moreover, since $k\ne 1$, we have $|(R_k\theta)_1 |^{2s}=|\theta_1|^{2s}$.
On the other hand, since exactly one of the two factors $\theta_i,\theta_j$ changes its sign, it follows that
\[
(R_k\theta)_i(R_k\theta)_j=-\theta_i\theta_j.
\]
Therefore, by using the change of variable $\theta\mapsto R_k\theta$, we get
\[
\int_{\mathbb S^{N-1}}\theta_i\theta_j|\theta_1|^{2s}\,d\theta = \int_{\mathbb S^{N-1}}(R_k\theta)_i(R_k\theta)_j
|(R_k\theta)_1|^{2s}\,d\theta = -\int_{\mathbb S^{N-1}}\theta_i\theta_j|\theta_1|^{2s}\,d\theta.
\]
This proves the claim in~\eqref{,lkjnhbgvfctdxrctfvgbhnm}.

For the case $i=j$, we set
\[
a:=\int_{\mathbb S^{N-1}}\theta_1^2|\theta_1|^{2s} d\theta.
\]
Moreover, we notice that the coordinates $\theta_2,\dots,\theta_N$ play a symmetric role. Indeed,  any permutation of these coordinates preserves the sphere, the surface measure, and the weight $|\theta_1|^{2s}$. Hence,
\[
\int_{\mathbb S^{N-1}}\theta_2^2|\theta_1|^{2s}\,d\theta = \cdots = \int_{\mathbb S^{N-1}}\theta_N^2|\theta_1|^{2s}\,d\theta=:b .
\]
Combining this with~\eqref{opkjoikjnio0ijohutfde4eswe4e5dstyh} and~\eqref{,lkjnhbgvfctdxrctfvgbhnm}, we obtain that
\begin{equation}\label{nbhgvfcdxcfvgybhnjm}
F_{e_1}(A) = \sum_{i=1}^N A_{ii}\int_{\mathbb S^{N-1}}\theta_i^2|\theta_1|^{2s}\,d\theta = a A_{11}+b\sum_{i=2}^N A_{ii}.
\end{equation}

Now, since
\[
A_{11}=e_1\cdot A e_1\quad\mbox{ and }\quad\sum_{i=2}^N A_{ii}=\tr(A)-e_1\cdot A e_1,
\]
we can write~\eqref{nbhgvfcdxcfvgybhnjm} as
\[
F_{e_1}(A) = b\,\tr(A)+(a-b)\,e_1\cdot A e_1.
\]
Equivalently,
\begin{equation}\label{dfvgbhnjkml,vhbnjkm,}
F_{e_1}(A) = \alpha\tr(A)+\beta\,e_1\cdot A e_1,
\end{equation}
where $\alpha:=b$ and~$\beta:=a-b$ depend only on $N$ and~$s$. 

We now determine $\alpha$ and~$\beta$. Recalling~\eqref{JNsdefn}, since $\theta\in\mathbb S^{N-1}$, we get
\begin{equation}\label{è+òplkoijhugyftrdeswasrdtfyguhj}
J_{N,s} = F_{e_1}(I) = N\alpha+\beta.
\end{equation}
On the other hand,  choosing~$A$ to be the matrix with entries $A_{ij}=\delta_{i, 1}\delta_{j, 1}$,  where $\delta_{ij}$ denotes the Kronecker delta. Thus, we have $\theta\cdot A\theta = |\theta_1|^2$, and hence
\begin{equation}\label{qaw3edr5tghyujkolò}
F_{e_1}(A) = \int_{\mathbb S^{N-1}}|\theta_1|^{2s+2}\,d\theta = \alpha+\beta.
\end{equation}
It only remains to compute the latter integral. By using spherical coordinates with respect to the
axis $e_1$, we have
\begin{equation}\label{xsdcfvgbhnjkml,.}
\begin{split}
&J_{N,s} = |\mathbb S^{N-2}|\int_{-1}^{1}|t|^{2s}(1-t^2)^{\frac{N-3}{2}}\,dt,\\
\mbox{ and }\quad&\int_{\mathbb S^{N-1}}|\theta_1|^{2s+2}\,d\theta = |\mathbb S^{N-2}|\int_{-1}^{1}|t|^{2s+2}(1-t^2)^{\frac{N-3}{2}}\,dt.
\end{split}
\end{equation}
By symmetry, it is enough to consider the integrals on $(0,1)$.  We set
\[
f(t):=t^{2s+1}(1-t^2)^{\frac{N-1}{2}}.
\]
We notice that $f(0)=f(1)=0$, and~$f'\in L^1(0,1)$. Thus, the Fundamental Theorem of Calculus yields
\[
0=f(1)-f(0)=\int_0^1 f'(t)\,dt = (2s+1)\int_0^1 t^{2s}(1-t^2)^{\frac{N-1}{2}}\,dt -(N-1)\int_0^1 t^{2s+2}(1-t^2)^{\frac{N-3}{2}}\,dt.
\]
By using the identity
\[
(1-t^2)^{\frac{N-1}{2}}=(1-t^2)^{\frac{N-3}{2}}-t^2(1-t^2)^{\frac{N-3}{2}},
\]
we conclude that
\[
(2s+1)\int_0^1 t^{2s}(1-t^2)^{\frac{N-3}{2}}\,dt = (2s+N)\int_0^1 t^{2s+2}(1-t^2)^{\frac{N-3}{2}}\,dt.
\]
Accordingly, by~\eqref{xsdcfvgbhnjkml,.}, we get
\[
\int_{\mathbb S^{N-1}}|\theta_1|^{2s+2}\,d\theta = \frac{2s+1}{N+2s}J_{N,s}.
\]
Combining this with~\eqref{qaw3edr5tghyujkolò},  we have that
\begin{equation}\label{lpojkhiu8ytfdrewaewsdrftgyhj}
\alpha+\beta=\frac{2s+1}{N+2s}J_{N,s}.
\end{equation}
Now, subctracting~\eqref{lpojkhiu8ytfdrewaewsdrftgyhj} from~\eqref{è+òplkoijhugyftrdeswasrdtfyguhj}, we get
\[
(N-1)\alpha = \frac{N-1}{N+2s}J_{N,s},
\]
namely
\[
\alpha=\frac{J_{N,s}}{N+2s}\quad\mbox{ and }\quad \beta = J_{N,s}-N\alpha = \frac{2s J_{N,s}}{N+2s}.
\]
Substituting these values in~\eqref{dfvgbhnjkml,vhbnjkm,}, we get
\begin{equation}\label{xsdcfvgrbtdhnyjhtbgrsvfdcsxa}
F_{e_1}(A) = \frac{J_{N,s}}{N+2s}\left(\tr(A)+2s\,e_1\cdot A e_1\right).
\end{equation}
This proves the desired claim~\eqref{pèkljniopjhgvftyrdszdwez} when $\omega=e_1$.

We now take $\omega\in\mathbb S^{N-1}$ and~$R_\omega\in O(N)$ such that 
\begin{equation}\label{àòèplkoijugytdreswasrdtfgyhj}
R_\omega e_1 =\omega.
\end{equation}
We set
\begin{equation}\label{vncmwqur385784SDFGH87r34y89to}
A_\omega:=R_\omega^{-1} A R_\omega = R_\omega^{T} A R_\omega.
\end{equation}
Thus, $A_\omega^T = A_\omega$. Hence, using the change of variable~$\theta := R_\omega\eta$, and recalling~\eqref{dretfyghuijoklpèòlppkoijuhgy},
\eqref{opkjoikjnio0ijohutfde4eswe4e5dstyh} and~\eqref{àòèplkoijugytdreswasrdtfgyhj},
\[
\begin{split}&
F_\omega(A)= \int_{\mathbb S^{N-1}}(R_\omega\eta\cdot A R_\omega\eta)|R_\omega e_1\cdot R_\omega\eta|^{2s}\,d\eta = \int_{\mathbb S^{N-1}}\eta\cdot (R_\omega^T A R_\omega)\eta \,|e_1 \cdot\eta|^{2s}\,d\eta\\
&\qquad\qquad= \int_{\mathbb S^{N-1}}\eta\cdot A_\omega \eta \,|\eta_1|^{2s}\,d\eta =  F_{e_1}(A_\omega).
\end{split}
\]
From this and~\eqref{xsdcfvgrbtdhnyjhtbgrsvfdcsxa}, we infer tht
\begin{equation}\label{cnewuti4uitkjt47838roihjsdfgfjie54oyu}
F_\omega(A)= \frac{J_{N,s}}{N+2s}\left(\tr(A_\omega)+2s\,e_1\cdot A_\omega e_1\right).
\end{equation}

Now we notice that, by~\eqref{àòèplkoijugytdreswasrdtfgyhj}
and~\eqref{vncmwqur385784SDFGH87r34y89to},
\[
e_1\cdot A_\omega e_1 =   e_1\cdot R_\omega^{T} A\,R_\omega e_1 = R_\omega e_1\cdot A\,R_\omega e_1 = \omega\cdot A \omega 
\]
and
\[
\begin{split}&
\tr(A_\omega)= \sum_{i=1}^N (R_\omega^T A R_\omega)_{ii} = \sum_{i,j,k=1}^N (R_\omega)_{ji} A_{jk} (R_\omega)_{ki}\\
&\qquad= \sum_{j,k=1}^N A_{jk} \sum_{i=1}^N (R_\omega)_{ji}(R_\omega)_{ki} = \sum_{j,k=1}^N A_{jk}\,\delta_{jk}= \sum_{j=1}^N A_{jj}= \tr(A).
\end{split}
\]
Combining these pieces of information with~\eqref{cnewuti4uitkjt47838roihjsdfgfjie54oyu}, we get
\[
F_\omega(A)= \frac{J_{N,s}}{N+2s}\left(\tr(A)+2s\,\omega\cdot A\omega\right),
\] which is
the claim in~\eqref{pèkljniopjhgvftyrdszdwez}.

Dividing by $J_{N,s}$ and recalling the definition in~\eqref{dretfyghuijoklpèòlppkoijuhgy}, we obtain the desired identity~\eqref{lsqkdJNCHUBSAKMOXxcsdjbhvhn}.
\end{proof}

We are now in the position to prove Proposition~\ref{thmmainconst}.
\begin{proof}[Proof of Proposition~\ref{thmmainconst}]
We consider the function $m:\mathbb S^{N-1}\to\R$ defined as \label{PALKMND:02p}
\begin{equation*}
m(\theta):=\frac{(N+2s)\theta\cdot A\theta-\tr(A)}{2s}.
\end{equation*}
We notice that $m\in L^1(\mathbb S^{N-1})$ and~$m(-\theta)=m(\theta)$.
Thus, recalling the setting of~\eqref{lòkjhgfdtresdtfyguhijokpl}
and using the notation~$\omega:=\frac{\xi}{|\xi|}$, for~$\xi\ne 0$, we infer that
\[
\begin{split}
\psi_m(\xi)&=\frac{c_{N,s}\kappa_s}{2s}|\xi|^{2s}\int_{\mathbb S^{N-1}}\big((N+2s)\theta\cdot A\theta-\tr(A)\big)
|\omega\cdot\theta|^{2s}\,d\theta\\
&=\frac{c_{N,s}\kappa_s}{2s}|\xi|^{2s}\left[(N+2s)\int_{\mathbb S^{N-1}}(\theta\cdot A\theta)|\omega\cdot\theta|^{2s}\,d\theta-\tr(A)\int_{\mathbb S^{N-1}}|\omega\cdot\theta|^{2s}\,d\theta\right].
\end{split}
\]
Now, by Lemmata~\ref{lemmauguale1} and~\ref{lemmaspherical}, we get $c_{N,s}\kappa_sJ_{N,s}=1$ and 
\[
(N+2s)
\int_{\mathbb S^{N-1}}(\theta\cdot A\theta)|\omega\cdot\theta|^{2s}\,d\theta= J_{N,s} \left(\tr(A)+2s\,\omega\cdot A\omega\right).
\]
Accordingly, by recalling 
the notation in~\eqref{ncewuri4y3itiyuwhdwuiefdefu5i654}, we conclude that
\[
\begin{split}
\psi_m(\xi)&=\frac{c_{N,s}\kappa_sJ_{N,s}}{2s}|\xi|^{2s}\left(\tr(A)+2s\,\omega\cdot A\omega-\tr(A)\right)=|\xi|^{2s}
\omega\cdot A\omega=|\xi|^{2s-2}\xi\cdot A\xi=p_A(\xi).
\end{split}
\]

Therefore,  combining Propositions \ref{BAsymbol} and~\ref{propEmsymbol}, we have
\[
\begin{split}
&\int_{\R^N}A D^s u(x)\cdot D^s v(x)\,dx = \int_{\R^N}p_A(\xi)\widehat{u}(\xi)\overline{\widehat{v}(\xi)}\,d\xi = \int_{\R^N}\psi_m(\xi)\,\widehat{u}(\xi)\,\overline{\widehat{v}(\xi)}\,d\xi\\
&\qquad=
\frac{c_{N,s}}{2}\iint_{\R^{2N}}\frac{(u(x)-u(y))(v(x)-v(y))}{|x-y|^{N+2s}}m\left(\frac{x-y}{|x-y|}\right)\,dx\,dy
\\&\qquad=\iint_{\R^{2N}}(u(x)-u(y))(v(x)-v(y))k_A(x,y)\,dx\,dy
,\end{split}
\] which entails that~\eqref{mnbhgvcfdxszxdfcgvhb} holds true,
as claimed.
\end{proof}

\section{Proof of  Theorem~\ref{corcharactA}}\label{lpkoijuhgyftdresdtfyguhijhb213}
We are now in the position to prove Theorem~\ref{corcharactA}.

\begin{proof}[Proof of Theorem~\ref{corcharactA}]
Assume first that~$(i)$ holds. By Proposition~\ref{thmmainconst}, the kernel
\[
\widetilde{k}_A(x,y):=\frac{c_{N,s}}{4s|x-y|^{N+2s}}\left((N+2s)\frac{(x-y)\cdot A(x-y)}{|x-y|^2}-\tr(A)\right)
\]
is symmetric and satisfies, for every~$u,v\in\mathcal S(\R^N)$,
\[
\int_{\R^N} A D^s u(x)\cdot D^s v(x)\, dx = \iint_{\R^{2N}}(u(x)-u(y))(v(x)-v(y))\widetilde{k}_A(x,y) \,dx\, dy
\]
 Thus, $k_A$ and~$\widetilde{k}_A$ represent the same bilinear form. By the uniqueness statement in Lemma~\ref{lemmauniq}, we infer that
\[
k_A(x,y)=\widetilde{k}_A(x,y)\quad\mbox{ for a.e.  }(x,y)\in\R^{2N}\setminus\Delta.
\]
Accordingly,~$(ii)$ is satisfied.

Conversely, assume that~$(ii)$ holds. Then, by Proposition~\ref{thmmainconst}, $k_A$ satisfies \eqref{mnbhgvcfdxszxdfcgvhb} for every~$u,v\in\mathcal S(\R^N)$,
which entails~$(i)$.

We now assume that either~$(i)$ or~$(ii)$ holds true and we focus on
the bounds of the kernel. We define
\begin{equation}\label{+èpokijuhgytfdreseasdfghjkl}
m_A(\theta):=\frac{c_{N,s}}{4s}\left((N+2s)\theta\cdot A\theta-\tr(A)\right),\quad{\mbox{ for every }}\theta\in\mathbb S^{N-1}.
\end{equation}
In this way, recalling~$(ii)$, we can write
\begin{equation}\label{cnmqr732iiSDFGHe7632rtgy}
k_A(x,y) = \frac{1}{|x-y|^{N+2s}} m_A\left(\frac{x-y}{|x-y|}\right)\quad\mbox{for a.e. }x\ne y.
\end{equation}

Let $\lambda_{\min}(A)$ and~$\lambda_{\max}(A)$ denote, respectively, the smallest and largest eigenvalues of the matrix~$A$. Thus,
\[
\lambda_{\min}(A)\le\theta\cdot A\theta\le\lambda_{\max}(A)\quad\mbox{ for every }\theta\in\mathbb S^{N-1}.
\]
Accordingly, we have that, for every~$\theta\in\mathbb S^{N-1}$,
\begin{equation}\label{+èàòplokiu8y7t6r5edsxdcfgvhb}
\frac{c_{N,s}}{4s}\left((N+2s)\lambda_{\min}(A)-\tr(A)\right)\le m_A(\theta)\le\frac{c_{N,s}}{4s}\left((N+2s)\lambda_{\max}(A)-\tr(A)\right).
\end{equation}
Furthermore, if $e_{\min},e_{\max}\in\mathbb S^{N-1}$ are unit
eigenvectors corresponding to the eigenvalues $\lambda_{\min}(A)$
and~$\lambda_{\max}(A)$, respectively, then
\[
\begin{split}
&m_A(e_{\min})=\frac{c_{N,s}}{4s}\left((N+2s)\lambda_{\min}(A)-\tr(A)\right)\\
\mbox{ and }\quad&m_A(e_{\max}) = \frac{c_{N,s}}{4s}\left((N+2s)\lambda_{\max}(A)-\tr(A)\right).
\end{split}
\]
Hence,
\begin{equation}\label{pèoiuhygftdreswaszdxcfgvhj}\begin{split}
&\min_{\theta\in\mathbb S^{N-1}}m_A(\theta) = \frac{c_{N,s}}{4s}\left((N+2s)\lambda_{\min}(A)-\tr(A)\right)\\ 
\mbox{ and }\quad&\max_{\theta\in\mathbb S^{N-1}}m_A(\theta) = \frac{c_{N,s}}{4s}\left((N+2s)\lambda_{\max}(A)-\tr(A)\right).
\end{split}\end{equation}

Now, since~$\|m_A\|_{L^{+\infty}(\mathbb S^{N-1})}<+\infty$,
we deduce from~\eqref{cnmqr732iiSDFGHe7632rtgy} that
\[
k_A(x,y)\le |k_A(x,y)|\le\|m_A\|_{L^{+\infty}(\mathbb S^{N-1})}|x-y|^{-N-2s}\quad\mbox{ for a.e.  }x\ne y.
\] Namely,
\eqref{srdtfyguioplèòlkoijhugyftr2} is satisfied with any~$\eta\le2s$ and any~$M\ge \|m_A\|_{L^{+\infty}(\mathbb S^{N-1})}$.

Now we suppose that~\eqref{tracecondition} holds
and prove~\eqref{srdtfyguioplèòlkoijhugyftr1}.
For this, we set
\[
\lambda_0:=\frac{c_{N,s}}{4s}\left((N+2s)\lambda_{\min}(A)-\tr(A)\right)\quad\mbox{ and }\quad\Lambda_0:=\frac{c_{N,s}}{4s}\left((N+2s)\lambda_{\max}(A)-\tr(A)\right).
\]
By~\eqref{tracecondition}, we have that~$\lambda_0>0$. Furthermore, by~\eqref{+èàòplokiu8y7t6r5edsxdcfgvhb},
\begin{equation}\label{nbhjgvfcdxszeawezsdrftgyhj}
\lambda_0\le m_A(\theta)\le\Lambda_0\quad\mbox{ for every }\theta\in\mathbb S^{N-1}.
\end{equation}
Combining this with~\eqref{cnmqr732iiSDFGHe7632rtgy}, we obtain that
\[
\lambda_0\le k_A(x,y)|x-y|^{N+2s}\le\Lambda_0\quad\mbox{ for a.e. }x\ne y.
\]
Thus,~\eqref{srdtfyguioplèòlkoijhugyftr1} holds with~$
\alpha=2s$, $\lambda=\lambda_0$ and~$\Lambda=\Lambda_0$.

It remains to prove the converse implication. Assume that a representing kernel $k_A$ satisfies~\eqref{srdtfyguioplèòlkoijhugyftr1} for some
$\alpha\in(0,2)$ and~$0<\lambda\le\Lambda<+\infty$. Since~\eqref{srdtfyguioplèòlkoijhugyftr2} is satisfied, for instance with~$\eta=2s$, Corollary~\ref{corkernelbounds} applies and yields $\alpha=2s$ and 
\begin{equation}\label{èplkojihugyftdrtextcftvgbh}
m_A^{*}(\theta)\ge\lambda>0\quad\mbox{ for a.e.  }\theta\in\mathbb S^{N-1},
\end{equation}
where $m_A^{*}$ denotes the function provided by Proposition~\ref{thmkernelstruct}.

On the other hand, Proposition~\ref{thmmainconst} provides the explicit
kernel
\[
\widetilde{k}_A(x,y)=\frac{c_{N,s}}{4s|x-y|^{N+2s}}\left((N+2s)\frac{(x-y)\cdot A(x-y)}{|x-y|^2}-\tr(A)\right) =\frac{1}{|x-y|^{N+2s}}m_A\left(\frac{x-y}{|x-y|}\right),
\]
with $m_A$ as in~\eqref{+èpokijuhgytfdreseasdfghjkl}. By the uniqueness statement in Proposition~\ref{thmkernelstruct}, we have
\[
k_A(x,y)=\widetilde{k}_A(x,y)\quad\mbox{ for a.e.  }x\ne y.
\]
Accordingly, $m_A^{*}$ coincides almost everywhere on $\mathbb S^{N-1}$ with the function $m_A$ defined in~\eqref{+èpokijuhgytfdreseasdfghjkl}. Thus,~\eqref{èplkojihugyftdrtextcftvgbh} gives
\[
m_A(\theta)\ge\lambda>0\quad\mbox{ for a.e.  }\theta\in\mathbb S^{N-1}.
\]
Since $m_A$ is continuous on $\mathbb S^{N-1}$, the same inequality holds for every~$\theta\in\mathbb S^{N-1}$. Hence,
\[
\min_{\theta\in\mathbb S^{N-1}}m_A(\theta)>0.
\]
By~\eqref{pèoiuhygftdreswaszdxcfgvhj}, this is equivalent to
\[
(N+2s)\lambda_{\min}(A)>\tr(A), 
\]
which is~\eqref{tracecondition}. 
\end{proof}

\section{A class of nonconstant matrices inducing a constant coefficient kernel}\label{MS-34or}

This section is devoted to the proof of Theorem \ref{invisiblepert}.

\begin{proof}[Proof of Theorem \ref{invisiblepert}]
We first consider the case $N=1$. In this case, the condition \eqref{lpèkoijhugyftdrxcfv} implies that $C(x)=0$ for every~$x\in\R$. Hence, $A(x)=A_0$ for every
$x\in\R$, and therefore
\[
\int_{\R} A(x)D^su(x)D^sv(x)\, dx = \int_{\R} A_0D^su(x)D^sv(x) \,dx
\] and
\[
k_{A_0}(x,y)
=
\frac{c_{1,s}A_0}{|x-y|^{1+2s}},
=
\frac{c_{1,s}}{2s|x-y|^{1+2s}}\left((1+2s)\frac{(x-y)\cdot A_0(x-y)}{|x-y|^2}-\tr(A_0)\right).
\]

We now assume $N\ge 2$. We set
\[
q_j(\xi):=\xi_j|\xi|^{s-1}\quad\mbox{ for } j=1,\ldots,N,
\]
and
\[
I_C[u,v]:=\int_{\R^N} C(x)D^su(x)\cdot D^sv(x)\,dx = \sum_{j,k=1}^N \ \int_{\R^N} C_{jk}(x)D_k^su(x)D_j^sv(x)\,dx.
\]

We now recall the following standard equalities for~$f$, $g\in \mathcal S(\R^N)$:
\[
\widehat{fg}(\zeta)=(2\pi)^{-N/2}\int_{\R^N}\widehat f(\eta)\widehat g(\zeta-\eta)\,d\eta\quad\mbox{ and }\quad \int_{\R^N}f(x)g(x)\,dx = \int_{\R^N}\widehat f(\zeta)\widehat g(-\zeta)\,d\zeta.
\]
Therefore, by using~\eqref{defnfrgradient}, and since
\[
\widehat{D_j^sv}(-\zeta) = i q_j(-\zeta)\widehat v(-\zeta)= -i q_j(\zeta)\widehat v(-\zeta),
\]
we obtain that
\[
\begin{split}
I_C[u,v]&= \sum_{j,k=1}^N \ \int_{\R^N}\widehat{C_{jk}D_k^su}(\zeta)\widehat{D_j^sv}(-\zeta)\,d\zeta\\& =(2\pi)^{-N/2}\sum_{j,k=1}^N \ \iint_{\R^{2N}}\widehat C_{jk}(\eta)\widehat{D_k^su}(\zeta-\eta)\,
\widehat{D_j^sv}(-\zeta)\,d\eta\,d\zeta \\
&= (2\pi)^{-N/2}\sum_{j,k=1}^N \ \iint_{\R^{2N}}\widehat C_{jk}(\eta)\, q_k(\zeta-\eta)q_j(\zeta)\widehat u(\zeta-\eta)
\widehat v(-\zeta)
\,d\eta\,d\zeta.
\end{split}
\]
We notice that all the above integrals are absolutely convergent. Indeed, $C\in C_c^{\infty}(\R^N, \R^{N\times N})$ implies that $\widehat C_{jk}\in\mathcal S(\R^N)$, while $u,v\in\mathcal S(\R^N)$ and~$|q(\xi)|=|\xi|^s$.  Hence, the polynomial growth of the factors $q_k(\zeta-\eta)$ and~$q_j(\zeta)$ is dominated by the rapid decay of $\widehat C_{jk}$, $\widehat u$, and~$\widehat v$. Thus,  Fubini's theorem applies.

We now fix $\eta,\zeta\in\R^N$ and we set $a:=\zeta-\eta$. Then,
\begin{equation}\label{awesrdtfyguiopè+òlpkojihughf}
\sum_{j,k=1}^N\widehat C_{jk}(\eta)q_k(a)q_j(\zeta) = |a|^{s-1}|\zeta|^{s-1} \sum_{j,k=1}^N\widehat C_{jk}(\eta)a_k\zeta_j = |a|^{s-1}|\zeta|^{s-1}\zeta\cdot \widehat C(\eta)a.
\end{equation}

Moreover, from~\eqref{lpèkoijhugyftdrxcfv}, we get
\[
\widehat C(\eta)^T=-\widehat C(\eta)\quad\mbox{ for every }\eta\in\R^N.
\]
Hence, for every~$a\in\R^N$, 
\begin{equation}\label{àòlkjhioiujhgfrtesawerd}
a\cdot \widehat C(\eta)a=0.
\end{equation}

In addition, by taking the Fourier transform of~\eqref{òldKMAFJNBHCJSNKM,}, we obtain that
\[
i\sum_{j=1}^N\eta_j\widehat C_{jk}(\eta)=0,\quad k=1,\ldots,N,
\]
namely
\begin{equation}\label{è+poijhgfdsrawe5rtyuijoo}
\eta^T\widehat C(\eta)=0.
\end{equation}
Since $\zeta=a+\eta$,  it follows from~\eqref{àòlkjhioiujhgfrtesawerd} and~\eqref{è+poijhgfdsrawe5rtyuijoo} that
\[
\zeta\cdot\widehat C(\eta)a = (a+\eta)\cdot\widehat C(\eta)a 
= a\cdot \widehat C(\eta)a + \eta\cdot \widehat C(\eta)a 
=a\cdot \widehat C(\eta)a + \eta^T \widehat C(\eta)\cdot a 
= 0.
\]
By using this information in~\eqref{awesrdtfyguiopè+òlpkojihughf},
we obtain that, for every~$\eta,\zeta\in\R^N$,
 $$\sum_{j,k=1}^N\widehat C_{jk}(\eta)q_k(a)q_j(\zeta)=0$$
and accordingly
\[
I_C[u,v]=0.
\]
Since $A(x)=A_0+C(x)$, this gives~\eqref{òè+poiuhygftdrszedrftychvghib}.

Finally,  set
\[
B_A[u,v]:=\int_{\R^N} A(x)D^su(x)\cdot D^sv(x)\,dx.
\]
By~\eqref{òè+poiuhygftdrszedrftychvghib}, we deduce that $B_A[u,v]=B_{A_0}[u,v]$ for every~$u,v\in\mathcal S(\R^N)$.
By Proposition~\ref{thmmainconst}, the kernel
\[
k_{A_0}(x,y)=\frac{c_{N,s}}{2s|x-y|^{N+2s}}\left((N+2s)\frac{(x-y)\cdot A_0(x-y)}{|x-y|^2}-\tr(A_0)\right)
\]
represents $B_{A_0}$, and therefore also $B_A$. Hence, if $k_A$ is any symmetric kernel representing $B_A$, the uniqueness of the representing kernel gives
\[
k_A(x,y)=k_{A_0}(x,y)\quad\mbox{for a.e. }x\ne y.
\]
This concludes the proof.
\end{proof}

\section{Examples}\label{MS-34or2}
The role of condition~\eqref{tracecondition} is worth emphasizing. We stress that~\eqref{tracecondition} is not needed for either
the existence or the uniqueness of the representing kernel, which follow from Propositions~\ref{thmkernelstruct} and~\ref{thmmainconst}. Rather, it ensures that $m_A(\theta)$ is uniformly bounded away from zero. Accordingly, it yields a corresponding lower bound for $k_A(x,y)$.

Corollary~\ref{corkernelbounds} shows that the kernel bounds~\eqref{srdtfyguioplèòlkoijhugyftr1} and~\eqref{srdtfyguioplèòlkoijhugyftr2} imply the uniform positivity of~$m_A(\theta)$, while Theorem~\ref{corcharactA} translates this condition into a condition on the matrix~$A$.

\smallskip

We now present two explicit examples illustrating the role of condition~\eqref{tracecondition}. The two matrices will be chosen from the same family with $1$ parameter.

Let $N=3$, $s\in(0,1)$, and set
\[
q:=\frac{1}{\sqrt{3}}(1,1,1).
\]
Consider the matrix
\[
P:=\frac{1}{3}
\begin{pmatrix}
2 & -1 & -1\\
-1 & 2 & -1\\
-1 & -1 & 2
\end{pmatrix}.
\]
Notice that, for every~$x\in\R^3$, we have $Px=x-(x\cdot q)q$, namely $P$ is the orthogonal projection onto $q^\perp$.

For $a\ge 0$, we define
\[
A_a:=I+aP =
\begin{pmatrix}
1+\frac{2a}{3} & -\frac{a}{3} & -\frac{a}{3}\\[5pt]
-\frac{a}{3} & 1+\frac{2a}{3} & -\frac{a}{3}\\[5pt]
-\frac{a}{3} & -\frac{a}{3} & 1+\frac{2a}{3}
\end{pmatrix}.
\]
We have that $A_a$ is symmetric. Moreover, if $a>0$, then $A_a$ is not diagonal. 

Now, since $Pq=0$ and
\[
P\xi=\xi\quad\mbox{ for every }\xi\in q^\perp,
\]
we infer that
\[
A_aq=q\quad\mbox{ and }\quad A_a\xi=(1+a)\xi\quad\mbox{ for every }\xi\in q^\perp.
\]
Accordingly, we have
\begin{equation}\label{pokijuhgyftdresawsrdtrfyguhijk}
\lambda_{\min}(A_a)=1,\quad\lambda_{\max}(A_a)=1+a,\quad\tr(A_a)=3+2a.
\end{equation}
Therefore,
\[
(3+2s)\lambda_{\min}(A_a)-\operatorname{tr}(A_a) = 2(s-a).
\]
In particular,~\eqref{tracecondition} is verified for $a<s$.

For a later use, we also compute the function $m_{A_a}(\theta)$. To this end, we notice that, for~$\theta\in\mathbb S^2$,
\[
\theta\cdot P\theta = |\theta|^2-(\theta\cdot q)^2 = 1-(\theta\cdot q)^2.
\]
Hence
\[
\theta\cdot A_a\theta = 1+a\left(1-(\theta\cdot q)^2\right).
\]

The kernel associated with $A_a$ is given by
\[
k_{A_a}(x,y) = \frac{1}{|x-y|^{3+2s}} m_{A_a}\left(\frac{x-y}{|x-y|}\right),
\]
where
\begin{equation}\label{drsetyguiokplèò+pooihugyhuijbkh}
m_{A_a}(\theta)=\frac{c_{3,s}}{2s}\Big((3+2s)\left(1+a(1-(\theta\cdot q)^2)\right)-(3+2a)\Big).
\end{equation}
Moreover, Proposition~\ref{thmmainconst} gives that, for every~$u,v\in\mathcal S(\R^3)$,
\begin{equation}\label{òèqSDlp<bfkoxcgijndskmaòl,AD}
\int_{\R^3}A_aD^su(x)\cdot D^sv(x)\,dx = \iint_{\R^3\times\R^3}(u(x)-u(y))(v(x)-v(y)) k_{A_a}(x,y)\,dx\,dy.
\end{equation}
Thus, the representation formula holds throughout the whole family $(A_a)_{a\ge 0}$, independently of whether the lower ellipticity bound~\eqref{srdtfyguioplèòlkoijhugyftr1} is satisfied.

\subsection{Representation with a uniform lower bound}
We take $a:=\frac{s}{2}$. Thus, we have
\[
A_{s/2}=
\begin{pmatrix}
1+\frac{s}{3} & -\frac{s}{6} & -\frac{s}{6}\\[5pt]
-\frac{s}{6} & 1+\frac{s}{3} & -\frac{s}{6}\\[5pt]
-\frac{s}{6} & -\frac{s}{6} & 1+\frac{s}{3}
\end{pmatrix}.
\]
By~\eqref{pokijuhgyftdresawsrdtrfyguhijk}, we get
\[
\lambda_{\min}(A_{s/2})=1,\quad\lambda_{\max}(A_{s/2})=1+\frac{s}{2},\quad\tr(A_{s/2})=3+s.
\]
Therefore, we obtain that
\[
(3+2s)\lambda_{\min}(A_{s/2})-\tr(A_{s/2}) = (3+2s)-(3+s) = s>0,
\]
namely $A_{s/2}$ satisfies~\eqref{tracecondition} for every~$s\in(0,1)$. Now, substituting $a=s/2$ in~\eqref{drsetyguiokplèò+pooihugyhuijbkh}, we obtain that
\[
m_{A_{s/2}}(\theta)=\frac{c_{3,s}}{2s}\left((3+2s)\left(1+\frac{s}{2}\bigl(1-(\theta\cdot q)^2\bigr)\right)-(3+s)\right) = \frac{c_{3,s}}{2}\left(1+\frac{3+2s}{2}(1-(\theta\cdot q)^2)\right).
\]
Since $0\le 1-(\theta\cdot q)^2\le 1$, we have that, for every~$\theta\in\mathbb S^2$,
\[
\frac{c_{3,s}}{2}\le m_{A_{s/2}}(\theta)\le\frac{c_{3,s}}{2}\left(1+\frac{3+2s}{2}\right)=\frac{c_{3,s}(5+2s)}{4}
.\]
 Therefore, for every~$x\ne y$,
\[
\frac{c_{3,s}}{2}\le k_{A_{s/2}}(x,y)|x-y|^{3+2s}\le\frac{c_{3,s}(5+2s)}{4}
.\]
 We notice that the latter inequality holds not only for $|x-y|\le1$, but at every scale. Hence,~\eqref{srdtfyguioplèòlkoijhugyftr1} is satisfied with
\[
\alpha=2s,\quad\lambda=\frac{c_{3,s}}{2},\quad\Lambda=\frac{c_{3,s}(5+2s)}{4}.
\]
The corresponding kernel is given by
\[
k_{A_{s/2}}(x,y) = \frac{c_{3,s}}{2|x-y|^{3+2s}}\left(1+\frac{3+2s}{2}\left(1-\left(\frac{x-y}{|x-y|}\cdot q\right)^2\right)\right).
\]
By~\eqref{òèqSDlp<bfkoxcgijndskmaòl,AD}, for every~$u,v\in\mathcal S(\R^3)$,
\[
\int_{\R^3} A_{s/2} D^su(x)\cdot D^sv(x)\,dx=\iint_{\R^3\times\R^3}(u(x)-u(y))(v(x)-v(y))k_{A_{s/2}}(x,y)\,dx\,dy.
\]
In addition, we notice that
\[
k_{A_{s/2}}(x,y)
\le
\frac{c_{3,s}(5+2s)}{4}
|x-y|^{-3-2s},
\]
so that~\eqref{srdtfyguioplèòlkoijhugyftr2} is satisfied, for instance, with
\[
\eta=2s,\quad M=\frac{c_{3,s}(5+2s)}{4}.
\]

\subsection{Representation without a uniform lower bound}
We now take $a:=s$. Thus, we have
\[
A_s=
\begin{pmatrix}
1+\frac{2s}{3} & -\frac{s}{3} & -\frac{s}{3}\\[5pt]
-\frac{s}{3} & 1+\frac{2s}{3} & -\frac{s}{3}\\[5pt]
-\frac{s}{3} & -\frac{s}{3} & 1+\frac{2s}{3}
\end{pmatrix}.
\]
Again, by~\eqref{pokijuhgyftdresawsrdtrfyguhijk}, we get
\[
\lambda_{\min}(A_s)=1,\quad\lambda_{\max}(A_s)=1+s,\quad\tr(A_s)=3+2s.
\]
Hence
\[
(3+2s)\lambda_{\min}(A_s) = 3+2s = \tr(A_s),
\]
and therefore~\eqref{tracecondition} does not hold. Now, by substituting $a=s$ in~\eqref{drsetyguiokplèò+pooihugyhuijbkh}, we obtain that
\[
m_{A_s}(\theta)=\frac{c_{3,s}}{2s}\left((3+2s)\left(1+s(1-(\theta\cdot q)^2)\right)-(3+2s)\right) = \frac{c_{3,s}(3+2s)}{2}(1-(\theta\cdot q)^2).
\]
Thus
\begin{equation}\label{w435e67890'ìpèlèkojhgfdxz}
0\le m_{A_s}(\theta)\le\frac{c_{3,s}(3+2s)}{2},
\end{equation}
and, in particular, $m_{A_s}(q)=m_{A_s}(-q)=0$.  Therefore,
\[
\min_{\theta\in\mathbb S^2}m_{A_s}(\theta)=0.
\]
We claim that this prevents~\eqref{srdtfyguioplèòlkoijhugyftr1} from holding. Indeed, suppose by contradiction that
\eqref{srdtfyguioplèòlkoijhugyftr1} holds for some $\alpha\in(0,2)$ and~$0<\lambda\le\Lambda<+\infty$. By Corollary~\ref{corkernelbounds},  one has~$
\alpha=2s$.
Hence the lower bound in~\eqref{srdtfyguioplèòlkoijhugyftr1} becomes
\[
\lambda\le k_{A_s}(x,y)|x-y|^{3+2s}
\]
for almost every~$x\ne y$ with $|x-y|\le1$. Now, since
\[
k_{A_s}(x,y)|x-y|^{3+2s} = m_{A_s}\left(\frac{x-y}{|x-y|}\right),
\]
this would imply that
\begin{equation}\label{cnweruig4387t547ewfguoi}
m_{A_s}(\theta)\ge\lambda\quad\mbox{ for a.e. }\theta\in\mathbb S^2.
\end{equation}

On the other hand, we have $m_{A_s}(q)=0$.  Since $m_{A_s}$ is continuous on $\mathbb S^2$, for every~$\lambda>0$ there exists a nonempty open neighborhood $U\subset\mathbb S^2$ of $q$ such that
\[
m_{A_s}(\theta)<\lambda\quad\mbox{ for every }\theta\in U.
\]
Since every nonempty open subset of $\mathbb S^2$ has positive surface measure, this contradicts the lower bound~\eqref{cnweruig4387t547ewfguoi}.

Nevertheless, the representation formula remains valid. Indeed, the kernel given by Proposition~\ref{thmmainconst} is
\[
k_{A_s}(x,y) = \frac{c_{3,s}(3+2s)}{2|x-y|^{3+2s}}\left(1-\left(\frac{x-y}{|x-y|}\cdot q\right)^2\right),
\]
and, for every~$u,v\in\mathcal S(\R^3)$,
\[
\int_{\R^3} A_s D^su(x)\cdot D^sv(x)\,dx = \iint_{\R^3\times\R^3}(u(x)-u(y))(v(x)-v(y))k_{A_s}(x,y)\,dx\,dy.
\]
Moreover, by~\eqref{w435e67890'ìpèlèkojhgfdxz}, we have that
\[
k_{A_s}(x,y)\le\frac{c_{3,s}(3+2s)}{2}|x-y|^{-3-2s}.
\]
Thus, ~\eqref{srdtfyguioplèòlkoijhugyftr2} holds with
\[
\eta=2s,\quad M=\frac{c_{3,s}(3+2s)}{2}.
\]
So this is a case where the representation formula and~\eqref{srdtfyguioplèòlkoijhugyftr2} remain true, while the positive lower bound in
\eqref{srdtfyguioplèòlkoijhugyftr1} fails.

\begin{appendix}

\section{Proof of Lemma~\ref{prop:E1normalization}}\label{appe:E1normalization}

\begin{proof}[Proof of Lemma~\ref{prop:E1normalization}]
We claim that, for $\xi\in\R^N$,
\begin{equation}\label{pokjhgfdre4356789i0opl}
\int_{\R^N}\frac{1-\cos(h\cdot\xi)}{|h|^{N+2s}}\, dh = \frac{|\xi|^{2s}}{c_{N,s}},
\end{equation}
Indeed, by \cite[3.381, n.4,  page 348]{MR3307944} we get
\[
|h|^{-N-2s} = \frac{1}{\Gamma(\frac N2+s)} \int_0^{+\infty} t^{\frac N2+s-1}e^{-t|h|^2}\, dt.
\]
Moreover, \cite[3.433, page 363]{MR3307944}, here applied with $n=1$ and $p=-s$, gives
\[
\int_0^{+\infty}\frac{1-e^{-r}}{r^{1+s}} \,dr =\frac{\Gamma(1-s)}{s},
\]
By using the latter identity and the change of variable $r:=|\xi|^2/(4t)$, we obtain that, for $\xi\ne0$,
\[
\begin{split}&
\int_{\R^N}\frac{1-\cos(h\cdot\xi)}{|h|^{N+2s}} \,dh = \frac{\pi^{N/2}}{\Gamma(\frac N2+s)} \int_0^{+\infty} t^{s-1}
\left(1-e^{-\frac{|\xi|^2}{4t}}\right) \,dt = \frac{\pi^{N/2}|\xi|^{2s}}{2^{2s}\Gamma(\frac N2+s)}\int_0^{+\infty}\frac{1-e^{-r}}{r^{1+s}} \,dr\\
&\qquad\qquad=\frac{\pi^{N/2}\Gamma(1-s)}{2^{2s}s\,\Gamma(\frac N2+s)}|\xi|^{2s}=\frac{|\xi|^{2s}}{c_{N,s}},
\end{split}
\]
which proves \eqref{pokjhgfdre4356789i0opl}.

Now, by Plancherel's identity, the translation property of the Fourier transform and~\eqref{pokjhgfdre4356789i0opl}, we have
\[
\begin{split}
\mathcal E_1(u,v)&=\frac{c_{N,s}}{2}\int_{\R^N}\frac{1}{|h|^{N+2s}}\int_{\R^N}(u(x)-u(x-h))(v(x)-v(x-h))\,dx\,dh\\
&=\frac{c_{N,s}}{2}\int_{\R^N}\frac{1}{|h|^{N+2s}}\int_{\R^N}|1-e^{-ih\cdot\xi}|^2\widehat u(\xi)\overline{\widehat v(\xi)} \,d\xi\, dh\\
&=c_{N,s}\int_{\R^N}\widehat u(\xi)\overline{\widehat v(\xi)}\left(\,\int_{\R^N}\frac{1-\cos(h\cdot\xi)}{|h|^{N+2s}} \,dh
\right)\,d\xi \\&= \int_{\R^N}|\xi|^{2s}\widehat u(\xi)\overline{\widehat v(\xi)}\, d\xi.
\end{split}
\]
We point out that the use of Fubini's theorem is justified by \eqref{pokjhgfdre4356789i0opl} and the rapid decay
of both~$\widehat u$ and~$\widehat v$. Moreover, by \eqref{defnfrgradient} and the Plancherel's identity, we have that
\begin{eqnarray*}&&
\int_{\R^N}D^su\cdot D^sv dx = \sum_{j=1}^N\int_{\R^N}\widehat{D_j^su}(\xi)\overline{\widehat{D_j^sv}(\xi)} d\xi \\&&\qquad= \sum_{j=1}^N \int_{\R^N}\xi_j^2|\xi|^{2s-2}\widehat u(\xi)\overline{\widehat v(\xi)}\,d\xi = \int_{\R^N}|\xi|^{2s}\widehat u(\xi)\overline{\widehat v(\xi)} d\xi.
\end{eqnarray*}
Comparing the last two expressions yields \eqref{E1normalization}, as desired.
\end{proof}

\section{Proof of Proposition \ref{propsuppcompatto}}\label{appendixsuppcompatto}
Here we prove that representing kernels cannot have finite interaction range.
\begin{proof}[Proof of Proposition \ref{propsuppcompatto}]
Suppose by contradiction that there exists $R_0>0$ such that
\begin{equation}\label{cmeru4i574839xbhjsavjcfewuh}
{\mbox{$k_A(x,y)=0$ for all~$(x,y)\in\R^{2N}$ with~$|x-y|>R_0$.}}\end{equation}
Let $\phi\in C_c^{\infty}(\R^N)$, with $\phi\not\equiv 0$, and take $R\ge 1$. We define
\begin{equation}\label{òlpèokiugtfrdesw}
u_R(x):=\phi\left(\frac{x}{R}\right).
\end{equation}
Since the fractional gradient is homogeneous of degree $s$, we have that
\[
D^s u_R(x) = R^{-s}(D^s\phi)\left(\frac{x}{R}\right).
\]
Thus,
\[
\|D^s u_R\|_{L^2(\R^N)}^2 = R^{N-2s}\|D^s\phi\|_{L^2(\R^N)}^2.
\]
Moreover, since $A$ is uniformly elliptic and $\phi\not\equiv 0$, we infer that
\begin{equation}\label{lèpokijuhygftde43waser6fyguyhjo}
\int_{\R^N} A(x)D^s u_R(x)\cdot D^s u_R(x) dx\ge\lambda_A\|D^s u_R\|_{L^2(\R^N)}^2 = \lambda_A R^{N-2s}\|D^s\phi\|_{L^2(\R^N)}^2>0.
\end{equation}

We now estimate the right-hand side of
\eqref{mnbhgvcfdxszxdfcgvhb}. For any~$z\in\R^N$, we define
\[
G(z):=\int_{\R^N}|\phi(x+z)-\phi(x)|^2 dx.
\]
By the fundamental theorem of calculus, we have
\[
\phi(x+z)-\phi(x) = \int_0^1 z\cdot\nabla\phi(x+tz) dt.
\]
Hence, by using Minkowski inequality, we get that
\[
\|\phi(\cdot+z)-\phi\|_{L^2(\R^N)}\le |z|\|\nabla\phi\|_{L^2(\R^N)}.
\]
Therefore, we obtain that
\[
G(z)\le |z|^2\|\nabla\phi\|_{L^2(\R^N)}^2.
\]
Thus, from the latter inequality and \eqref{òlpèokiugtfrdesw}, we infer that, for every~$h\in\R^N$,
\begin{equation}\label{plkmk098uyt5ewq2wsdertyhjk}
\int_{\R^N} |u_R(x+h)-u_R(x)|^2 dx = R^N G\left(\frac{h}{R}\right)\le R^{N-2}|h|^2 \|\nabla\phi\|_{L^2(\R^N)}^2.
\end{equation}

We now set
\[
E_R:=\iint_{\R^{2N}}(u_R(x)-u_R(y))^2 k_A(x,y) \,dx\, dy
\] and we observe that, by~\eqref{cmeru4i574839xbhjsavjcfewuh},
\[
E_R=\iint_{\{|x-y|\le R_0\}}(u_R(x)-u_R(y))^2 k_A(x,y) \,dx\, dy.
\]

We distinguish two cases. If $R_0\le 1$, then $\{|x-y|\le R_0\}\subseteq\{|x-y|\le 1\}$, and therefore we can use assumption \eqref{poiuytresdxcvghbnjm}, together with~\eqref{plkmk098uyt5ewq2wsdertyhjk}, to find that
\begin{equation}\label{klmjnhgufyctdxrszea<wzesxrdctfvgbhnj}
\begin{split}
E_R &\le\Lambda\iint_{\{|x-y|\le R_0\}}\frac{|u_R(x)-u_R(y)|^2}{|x-y|^{N+\alpha}} \,dx\, dy\\
&\le\Lambda\int_{\{|h|\le R_0\}}\frac{1}{|h|^{N+\alpha}}\left(\;\int_{\R^N}|u_R(x+h)-u_R(x)|^2 \,dx\right)\,dh\\&\le C R^{N-2}\int_{\{|h|\le R_0\}}|h|^{2-N-\alpha} \,dh,
\end{split}
\end{equation}
where the last integral is finite since we assumed $\alpha\in(0,2)$. Thus, there exists a constant~$C_{R_0}>0$,
independent of $R$, such that
\begin{equation}\label{jhugyftrdeswdrtfygh}
E_R\le C_{R_0}R^{N-2}.
\end{equation}

If instead~$R_0>1$, we write~$E_R=E_{R,1}+E_{R,2}$,
where
\begin{eqnarray*}&&
E_{R,1}:=\iint_{\{|x-y|\le1\}}(u_R(x)-u_R(y))^2k_A(x,y)\,dx\,dy\\
\mbox{ and }\quad&& E_{R,2}:=\iint_{\{1<|x-y|\le R_0\}}(u_R(x)-u_R(y))^2k_A(x,y) \,dx\, dy.
\end{eqnarray*}
By using~\eqref{klmjnhgufyctdxrszea<wzesxrdctfvgbhnj} with $R_0=1$
we see that
\begin{equation}\label{ER1kjnbhgvfcdxsxdcfvgb}
E_{R,1}\le C R^{N-2},
\end{equation}
where $C>0$ is independent of $R$. 

Moreover, by \eqref{cxzsedrftyhujiko098765r4e3wq2a<}, we get
\begin{eqnarray*}
E_{R,2}&\le& M\iint_{\{1<|x-y|\le R_0\}}\frac{|u_R(x)-u_R(y)|^2}{|x-y|^{N+\eta}} \,dx\, dy
\\&\le& M\int_{\{1<|h|\le R_0\}}\frac{1}{|h|^{N+\eta}}\left(\;\int_{\R^N}|u_R(x+h)-u_R(x)|^2\, dx\right)\,dh.
\end{eqnarray*}
Thus, by \eqref{plkmk098uyt5ewq2wsdertyhjk}, we have
\begin{equation}\label{ER2mkjhgfdeswedrrfghjk}
E_{R,2}\le C R^{N-2}\int_{\{1<|h|\le R_0\}} |h|^{2-N-\eta}\, dh \le 
C_{R_0}R^{N-2},
\end{equation} for some~$C_{R_0}>0$ independent of~$R$.

Combining \eqref{ER1kjnbhgvfcdxsxdcfvgb} and \eqref{ER2mkjhgfdeswedrrfghjk}, we obtain that
\begin{equation}\label{ERfinalokpijhugyftdreswrdtfgyh}
E_R\le C_{R_0}R^{N-2},
\end{equation}
where $C_{R_0}>0$ is independent of $R$. 

Hence, by \eqref{mnbhgvcfdxszxdfcgvhb}, \eqref{lèpokijuhygftde43waser6fyguyhjo} and \eqref{ERfinalokpijhugyftdreswrdtfgyh}, we infer that
\[
0<\lambda_A R^{N-2s}\|D^s\phi\|_{L^2(\R^N)}^2\le C_{R_0}R^{N-2}.
\]
Since $R\ge 1$, dividing both sides by $R^{N-2}$ yields
\[
\lambda_A \|D^s\phi\|_{L^2(\R^N)}^2 R^{2-2s}\le C_{R_0}.
\]
Thus, since $\lambda_A \|D^s\phi\|_{L^2(\R^N)}^2>0$ and $s\in(0,1)$, the left hand side tends to $+\infty$ as $R\to+\infty$, while $C_{R_0}$ is a constant independent of~$R$. This is a contradiction. 
\end{proof}

\section{Polar representation of homogeneous kernels}
We now show that a locally integrable kernel $K$ satisfying a homogeneity assumption, can be represented in polar coordinates in terms of its angular part. 

\begin{lemma}\label{lemmatecnico2.2}
Let $\gamma>0$ and~$N\ge 1$. Assume that $K\in L^1_{loc}(\R^N\setminus\{0\})$ satisfies, for every~$r>0$,
\begin{equation}\label{ò.,adklms<cgvysjnzklm,}
K(r\zeta)=r^{-\gamma}K(\zeta)\quad\mbox{for a.e. }\zeta\ne0.
\end{equation}

Then, there exists $m\in L^1(\mathbb S^{N-1})$ such that
\[
K(\rho\theta)=\rho^{-\gamma}m(\theta)\quad\mbox{ for a.e.  }(\rho,\theta)\in(0,+\infty)\times\mathbb S^{N-1}.
\]
\end{lemma}

\begin{proof}
Let $d\sigma$ denote the surface measure on $\mathbb S^{N-1}$, and take $0<\alpha<\beta<+\infty$.  When $N=1$, we use the notation $\mathbb S^0=\{-1,1\}$, endowed with the $0$ dimensional Hausdorff measure, so that the usual polar integration formula remains true.

We define
\begin{equation}\label{òàèù+plkoijhugyftdresdtfygh}
G(\rho,\theta):=\rho^\gamma K(\rho\theta),\quad\mbox{ for }(\rho,\theta)\in(0,+\infty)\times\mathbb S^{N-1}.
\end{equation}
We claim that
\begin{equation}\label{q23456789olkmnbvcdse4567uikmnb }
G\in L^1\left((\alpha,\beta)\times\mathbb S^{N-1}\right).
\end{equation}
Indeed, set
\[
C_{\alpha,\beta}:=\max_{\rho\in[\alpha,\beta]}\rho^{\gamma-N+1}<+\infty.
\]
Since $\{x\in\R^N:\alpha\le |x|\le\beta\}\subset\R^N\setminus\{0\}$ is compact, and~$K\in L^1_{loc}(\R^N\setminus\{0\})$,  by using the change of variable $x:=\rho\theta$, we infer that
\[
\begin{split}&
\int_\alpha^\beta\int_{\mathbb S^{N-1}}|G(\rho,\theta)| d\sigma(\theta) d\rho= \int_\alpha^\beta\int_{\mathbb S^{N-1}}\rho^\gamma |K(\rho\theta)| d\sigma(\theta) d\rho\le C_{\alpha,\beta}\int_\alpha^\beta\int_{\mathbb S^{N-1}}|K(\rho\theta)|\rho^{N-1}\,d\sigma(\theta)\,d\rho\\
&\qquad\qquad=\int_{\{x\in\R^N:\alpha\le |x|\le\beta\}} |K(x)|\,dx <+\infty,
\end{split}
\]
thus establishing~\eqref{q23456789olkmnbvcdse4567uikmnb }.

We now use the homogeneity assumption. Fix $r>0$. By~\eqref{ò.,adklms<cgvysjnzklm,} we get
\[
K(r\rho\theta)=r^{-\gamma}K(\rho\theta)\quad\mbox{ for a.e. } (\rho,\theta)\in(0,+\infty)\times\mathbb S^{N-1}.
\]
Accordingly, we have
\begin{equation}\label{òlkjihugyftdxrcfv}
G(r\rho,\theta)=(r\rho)^\gamma K(r\rho\theta)=r^\gamma\rho^\gamma r^{-\gamma}K(\rho\theta)=G(\rho,\theta)\quad\mbox{ for a.e. } (\rho,\theta)\in(0,+\infty)\times\mathbb S^{N-1}.
\end{equation}

We claim that, in fact,
\begin{equation}\label{.òlpkojihugyfctdxrszexdcfvgbhn}
G(r\rho,\theta)=G(\rho,\theta) \quad\mbox{ for a.e. } (r, \rho,\theta)\in(0,+\infty)\times (0, +\infty)\times\mathbb S^{N-1}.
\end{equation}
Indeed, define
\[
E:=\left\{(r,\rho,\theta)\in(0,+\infty)\times (0, +\infty)\times\mathbb S^{N-1} \mbox{ s.t. } G(r\rho,\theta)\ne G(\rho,\theta)\right\}.
\]
For every fixed $r>0$, let
\[
E_r:=\left\{(\rho,\theta)\in(0,+\infty)\times\mathbb S^{N-1} \mbox{ s.t. } (r,\rho,\theta)\in E\right\}.
\]
By~\eqref{òlkjihugyftdxrcfv}, it follows that $E_r$ has zero
measure in~$(0,+\infty)\times\mathbb S^{N-1} $ for every~$r>0$. Therefore,
by Tonelli's theorem applied to $\chi_E$, it follows that
\[
\int_0^{+\infty}\int_0^{+\infty}\int_{\mathbb S^{N-1}} \chi_E(r,\rho,\theta) d\sigma(\theta)\,d\rho\,dr =\int_0^{+\infty}\left[\int_0^{+\infty}\int_{\mathbb S^{N-1}}\chi_E(r,\rho,\theta) d\sigma(\theta)\,d\rho\right]dr=0.
\]
Thus, $E$ has measure zero. This proves the claim in~\eqref{.òlpkojihugyfctdxrszexdcfvgbhn}.

We now set~$s:=r\rho$ and we have from~\eqref{.òlpkojihugyfctdxrszexdcfvgbhn} that
\[
G(s,\theta)=G(\rho,\theta)\quad\mbox{ for a.e. } (s, \rho,\theta)\in(0,+\infty)\times (0, +\infty)\times\mathbb S^{N-1}.
\]
Applying Fubini's theorem once more, we deduce that for almost every~$(\rho,\theta)\in(0,+\infty)\times\mathbb S^{N-1}$, one has
\begin{equation}\label{kmjnbhgvcfdxszxfcghvbjnkml,ò.}
G(s,\theta)=G(\rho,\theta)\quad\mbox{for a.e. }s\in(\alpha,\beta).
\end{equation}

We now define
\[
m(\theta):=\frac1{\beta-\alpha}\int_\alpha^\beta G(s,\theta)\,ds
.\]
Thanks to~\eqref{q23456789olkmnbvcdse4567uikmnb }, we see that~$m\in L^1(\mathbb S^{N-1})$.

Finally,  by~\eqref{kmjnbhgvcfdxszxfcghvbjnkml,ò.}, we deduce that, for almost every~$(\rho,\theta)\in (0, +\infty)\times\mathbb S^{N-1}$, 
\[
m(\theta)= \frac1{\beta-\alpha}\int_\alpha^\beta G(\rho,\theta)\,ds=G(\rho,\theta).
\]
Recalling~\eqref{òàèù+plkoijhugyftdresdtfygh}, we obtain the desired result.
\end{proof}

\section{A further counterexample}\label{appDanielExample}
In this section, we provide an additional counterexample, shared with us by Daniel E. Spector.

Consider the energy
\[
\mathcal E_A^s(u,v):=\int_{\R^N} A D^s u\cdot D^s v = \mbox{PV}\iint_{\R^{2N}}\widetilde v(x)(\widetilde u(x)-\widetilde u(y))k_A(x,y) dy dx.
\]
The question of whether this energy defines a Dirichlet form is closely related to the regularity theory and to obstacle problems associated with the corresponding operator. Here, the fractional derivative $D^s u$ is defined componentwise by
\[
(D^s u)_j:=\frac{\partial}{\partial x_j} I_{1-s}u,
\]
where $I_s$ denotes the fractional integral, and
\[
k_A(x,y)=\mbox{PV}\int_{\R^N} A(z)\frac{y-z}{|y-z|^{d+s+1}}\cdot\frac{z-x}{|z-x|^{d+s+1}} dz.
\]
Consider now the one-dimensional case, where $A=a$. Then
\[
k_a(x,y) = \mbox{PV}\int_{\R} a(z)\frac{y-z}{|y-z|^{s+2}}\frac{z-x}{|z-x|^{s+2}} dz,
\]
and ellipticity and boundedness correspond to
\[
0<\varepsilon\le a\le C.
\]
If the functions involved in the bilinear pairing are compactly
supported, integrability at infinity does not play a role, and the
endpoint case $s=0$ can be considered as well.

For $s=0$, the partial fraction decomposition
\[
\frac{1}{(y-z)(z-x)} = \frac{1}{y-x}\left(\frac{1}{y-z} + \frac{1}{z-x}\right)
\]
gives
\[
k_a(x,y) = \frac{\pi}{y-x}\left(Ha(y)-Ha(x)\right),
\]
where $H$ denotes the Hilbert transform.

If $k_a$ were nonnegative for every strictly positive coefficient $a$, then, by approximation, the same property would hold for nonnegative coefficients. Consider therefore the limiting choice
\[
a=\chi_{[0,b]}.
\]
Since the Hilbert transform of the characteristic function of an interval can be computed explicitly, one obtains
\[
k_a(x,y) = \frac{1}{y-x}\left(\log\left|\frac{y}{y-b}\right| - \log\left|\frac{x}{x-b}\right|\right).
\]
Let $x<y<0$. For fixed $x<0$,
\[
\lim_{y\to0^-} k_a(x,y)=-\infty.
\]
Thus, the kernel is negative for $y$ sufficiently close to zero. On the other hand, as $y$ approaches $b$, the kernel becomes positive and arbitrarily large in magnitude. Hence, the change of sign cannot be attributed to a sign convention in the computation.
\end{appendix}

\section*{Acknowledgments}
The authors thank Daniel E. Spector for the helpful discussions and insights, and for sharing with them the counterexample presented in Appendix \ref{appDanielExample}.

The authors are members of the Australian Mathematical Society (AustMS). CS
is member of INdAM-GNAMPA.

This work has been supported by the Australian Laureate Fellowship FL190100081 and
the Australian Future Fellowship FT230100333.

CS acknowledges the support of the Juan de la Cierva Fellowship (grant number JDC2023-050365-I).

CS gratefully acknowledges the support of the Australian Laureate Fellowship FL190100081 for her research stay at UWA, where part of this research was carried out.

{\small \def\cdprime{$''$}
}
\vfill

\end{document}